\documentclass[11pt]{amsart}
\usepackage{amsmath,amsthm,amscd,amssymb, color}
\usepackage{mathrsfs}  
\usepackage{latexsym}
\usepackage[backref,hypertexnames=false]{hyperref}
\usepackage[alphabetic]{amsrefs}
\usepackage[capitalize, nameinlink]{cleveref}
\usepackage{enumitem}
\usepackage[a4paper,margin=1.25in]{geometry}
\usepackage{bbm}
\usepackage{mathtools}
\usepackage{algorithm}
\usepackage{algpseudocode}
\usepackage{tikz-cd}
\usepackage{colonequals}


\newcommand{\Q}{\mathbb Q}
\newcommand{\C}{\mathbb C}
\newcommand{\Z}{\mathbb Z}
\newcommand{\F}{\mathbb F}

\newcommand{\frakp}{\mathfrak p}

\newcommand{\PP}{\mathbb P}

\newcommand{\calO}{\mathcal O}

\newcommand{\bmx}{\left( \begin{matrix}}
\newcommand{\emx}{\end{matrix} \right)}

\newcommand{\poly}{\varphi} 
\newcommand{\polytwo}{\psi}
\newcommand{\iz}{is}

\DeclareMathOperator{\GL}{GL}
\DeclareMathOperator{\SL}{SL}
\DeclareMathOperator{\PGL}{PGL}

\DeclareMathOperator{\Nm}{Nm} 
 
\DeclareMathOperator{\End}{End} 
\DeclareMathOperator{\Pic}{Pic}

\DeclareMathOperator{\Jac}{Jac} 
 
\DeclareMathOperator{\Gal}{Gal} 
\DeclareMathOperator{\Aut}{Aut} 
 
\DeclareMathOperator{\diag}{diag}
\DeclareMathOperator{\Spec}{Spec}

\newcommand{\overbar}[1]{\mkern 1.5mu\overline{\mkern-1.5mu#1\mkern-1.5mu}\mkern 1.5mu}
\newcommand{\sslash}{\mathbin{/\mkern-6mu/}}
\usepackage{multirow}
\usepackage{array}

\newcommand{\PreserveBackslash}[1]{\let\temp=\\#1\let\\=\temp}
\newcolumntype{C}[1]{>{\PreserveBackslash\centering}p{#1}}

\numberwithin{algorithm}{section}
\newtheorem{lemma}[algorithm]{Lemma}
\newtheorem{prop}[algorithm]{Proposition}
\crefname{prop}{Proposition}{Propositions}
\newtheorem{thm}[algorithm]{Theorem}
\newtheorem{cor}[algorithm]{Corollary}

\newtheorem{question}[algorithm]{Question}

\theoremstyle{definition}
\newtheorem{ex}[algorithm]{Example}

\theoremstyle{remark}
\newtheorem{rem}[algorithm]{Remark}

\numberwithin{equation}{section}

\begin{document}

\title{Generic models for genus 2 curves \\ with real multiplication}
\author{Alex Cowan}
\address{Department of Mathematics, Harvard University, Cambridge, MA 02138 USA}
\email{cowan@math.harvard.edu}

\author{Sam Frengley}
\address{School of Mathematics, University of Bristol, Bristol, BS8 1UG, UK}
\email{sam.frengley@bristol.ac.uk}

\author{Kimball Martin}
\address{Department of Mathematics, University of Oklahoma, Norman, OK 73019 USA}
\email{kimball.martin@ou.edu}

\date{\today}

\begin{abstract}
  Explicit models of families of genus 2 curves with multiplication by $\sqrt D$
  are known for $D= 2, 3, 5$.  We obtain generic models for genus 2 curves
  over $\mathbb Q$ with real multiplication in 12 new cases, including all fundamental discriminants $D < 40$.
  A key step in our proof is to develop an algorithm for minim{\iz}ation of conic bundles fibred over $\PP^2$. We apply this algorithm to simplify the equations for the Mestre conic associated to the generic point on the Hilbert modular surface of fundamental discriminant $D < 100$ computed by Elkies--Kumar.
\end{abstract}

\maketitle

%%%%%%%%%%%%%%%%%%%%%%%%%%%%%%%%
%%%%%%%%%%%%%%%%%%%%%%%%%%%%%%%%
%
%  INTRODUCTION
%
%%%%%%%%%%%%%%%%%%%%%%%%%%%%%%%%
%%%%%%%%%%%%%%%%%%%%%%%%%%%%%%%%

\section{Introduction}
\label{sec:introduction}
Let $C$ be a genus 2 curve over a field $k$ of characteristic $0$ and let $D > 0$ be a fundamental discriminant. Let $\Jac(C)$ denote the Jacobian of $C$. We say $C$ has RM $D$ if it has real multiplication by the quadratic order $\calO_D$ of discriminant $D$, i.e., if $\calO_D$ embeds into the ring of endomorphisms of $\Jac(C)$ fixed by the Rosati involution.  Families of genus 2 curves with RM 5 and RM 8 have been known for some time (e.g., \cite{mestre:families,brumer:rank,Bending_paper}), however these families do not provide simple ways to parametr{\iz}e genus 2 curves with RM $D$ over $k$ (even up to twists, i.e., $\bar k$-isomorphism).  Moreover the methods used to construct these families are very specific to $D=5$ and $D=8$.

In this paper we develop a method to (i) give generic models for genus 2 curves with RM $D$, and (ii) parametr{\iz}e such curves up to twists.  We successfully carry out this method for many fundamental discriminants $D$ (including all 11 positive fundamental discriminants $D < 40$), under the assumption that $k$ has characteristic $0$.

\begin{thm} \label{thm:models}
  Let $D \in \{5, 8, 12, 13, 17, 21, 24, 28, 29, 33, 37, 44, 53, 61 \}$ and let $k$ be a field of characteristic $0$.
  Let $F_D(z, g, h, r, s; x) \in \Q(z,g,h,r,s)[x]$ be the sextic polynomial given in the electronic data associated to this paper \cite{electronic}.  
  Let $\mathscr{L}_D / \Q$ denote the conic bundle given by the vanishing of 
  \begin{equation*}
    z^2 - \lambda_D \quad \text{and} \quad r^2 - D s^2 - q_D
  \end{equation*}
  in $\mathbb{A}^5$, where $\lambda_D, q_D \in \Z[g,h]$ are the polynomials defined via \eqref{eq:EK-model} and \cref{tab:qD} respectively. Then:
  \begin{enumerate}[label=\textnormal{(\roman*)}]
  \item
    \label{iii:gen-fam}
    The family of genus $2$ curves given by the Weierstrass equations
    \begin{equation*}
      y^2 = F_D(z, g, h, r, s; x) 
    \end{equation*}
    for $(z,g,h, r, s) \in \mathscr{L}_D(k)$  provides a generic model (in the sense of \Cref{rem:genericity}) for genus $2$ curves with RM $D$ defined over $k$.

  \item
    Generically, two such curves are $\bar k$-isomorphic if and only if the corresponding points on $\mathscr{L}_D$ have the same image under the forgetful map $\mathscr{L}_D \to \mathbb{A}^2$ given by $(z,g,h,r,s) \mapsto (g,h)$.  
  \end{enumerate}
\end{thm}

\cref{thm:models}\ref{iii:gen-fam} provides families in 5 parameters satisfying 2 relations, but for many such $D$ one can do much better. Indeed, when $D \leq 17$ we show in \cref{sec:3-param-models} that the threefold $\mathscr{L}_D$ is rational over $\Q$. By parametrising $\mathscr{L}_D$ we give generic families in 3 parameters with no relations. In fact, when $D=17$ a generic family exists in 2 parameters with no relations. To illustrate this we present the families for $D=12$ and $17$.

\begin{cor} [A generic RM 12 family]
  \label{thm:RM12}
  %%%%%%%%%%%%%%%%%%%%%%%%%%%%%%%%
  % BY DISCRIMINANT 12
  %%%%%%%%%%%%%%%%%%%%%%%%%%%%%%%%
  Let $k$ be a field of characteristic $0$. For each  $a,b,c \in k$ consider the cubic $k$-algebra $L = k[r]/\xi(r)$ where
  \begin{equation*}
    \xi(r) = r^3 - 3(a^2 - 3 b^2) r + 2 (a^2 - 3 b^2),
  \end{equation*}
  and let
  \begin{flalign*}
    \poly(x) =\,\, & (r^2 + (2 a - 3 b) r - (a^2 + 2 a - 3 b^2 - 3 b c - 3 b)) x^2 - 6((a - 2 b) r - a c - a + 2 b)x \\
                   & - 3 (r^2 - (2 a - 3 b) r - (a^2 - 2 a - 3 b^2 + 3 b c + 3 b)).
  \end{flalign*}
  If $\xi(r)$ has no repeated roots we write $f_{12}(a,b,c;x) = \Nm_{L/K} \poly(x)$. If $f_{12}(a,b,c;x)$ has no repeated roots then the Jacobian $J$ of the genus $2$ curve $C/k$ with Weierstrass equation $C: y^2 =  f_{12}(a,b,c;x)$ has RM $12$ over $\bar k$, and if $\End_{\overbar k}(J)$ is abelian then the RM is defined over $k$.  Moreover, this is a generic family of genus $2$ curves with RM $12$ over $\Q$.
\end{cor}

\begin{cor} [A generic RM 17 family]
  \label{thm:RM17}
  %%%%%%%%%%%%%%%%%%%%%%%%%%%%%%%% 
  % BY DISCRIMINANT 17 
  %%%%%%%%%%%%%%%%%%%%%%%%%%%%%%%%
  Let $k$ be a field of characteristic $0$. For each $a,b \in k$ let
  \begin{flalign*}
    \poly(x)  &= (a^2 - 8 a b + 4 a - 9 b^2 - 6 b + 3) x^3 + 3 (7 a b - 3 a + 7 b^2 + 4 b - 3) x^2 + 4( a^2 - 7 a b   & \\
              &\quad + 3 a - 4 b^2) x + 4( 3 a b - a + b^2 - b),                                                      & \\
    \intertext{and}
    \polytwo(x) &= 4(a^2 b + 5 a b^2 - 7 a b + 2 a - 6 b^2 + 2 b) x^3 + 4 (6 a^2 b - 2 a^2 - 12 a b^2 + 11 a b - 3 a    & \\
              &\quad+ 14 b^2 - 6 b) x^2 + ( 4a^3 - 34 a^2 b + 16 a^2 + 38 a b^2 - 43 a b + 9 a - 43 b^2 + 36 b - 9) x & \\
              &\quad+ 12 a^2 b  - 4 a^2 - 10 a b^2 + 14 a b - 4 a + 11 b^2 - 14 b + 3.                                &
  \end{flalign*}
  If the sextic polynomial $f_{17}(a,b;x) = \poly(x)\polytwo(x)$ has no repeated roots, the Jacobian $J$ of the genus $2$ curve $C/k$ given by the Weierstrass equation $C: y^2 = f_{17}(a,b;x)$ has RM $17$ over $\bar k$, and if $\End_{\overbar k}(J)$ is abelian then the RM is defined over $k$. Moreover, this is a generic family of genus $2$ curves with RM $17$ over $\Q$.
\end{cor}

\begin{rem}
  \label{rmk:when-isom-abc}
  In our proof of Corollaries~\ref{thm:RM12} and \ref{thm:RM17} we obtain a map from the respective parameter spaces to the $(g,h)$-plane. This allows us to give simple conditions (exploiting \Cref{thm:models}) for when pairs of parameters correspond to (geometrically) isomorphic curves (see~\Cref{prop:when-isom}).
\end{rem}

\subsection{\texorpdfstring{Parametr{\iz}ing genus $2$ curves with RM $D$}{Parametrising genus 2 curves with RM D}}
Our approach to proving \cref{thm:models} is via moduli.  Let $D > 0$ be a fundamental discriminant. The Hilbert modular surface $Y_-(D)$ of discriminant $D$ provides a (coarse) moduli space for genus 2 curves with RM $D$ together with the RM $D$ action.
Forgetting the choice of RM $D$ real{\iz}es $Y_-(D)$ as a double cover of the Humbert surface of discriminant $D$, which we denote $\mathcal H_D$. (For precise definitions, see~\Cref{sec:real-mult-hilb}).

For each fundamental discriminant $D < 100$ the Humbert surface $\mathcal{H}_D$ is $\Q$-birational to $\mathbb{A}^2$ and in each of these cases Elkies and Kumar \cite{Elkies-Kumar} computed explicit rational parametrisations $\Q(\mathcal{H}_D) \cong \Q(g,h)$ together with rational maps $\mathbb{A}^2 \dashrightarrow \mathcal{M}_2$ real{\iz}ing the moduli interpretation of $\mathcal{H}_D$. Moreover, they give explicit $\Q$-birational models for $Y_-(D)$ in the form
\begin{equation} \label{eq:EK-model}
 z^2 = \lambda_D,
\end{equation}
for each fundamental discriminant $D < 100$, where $\lambda_D \in \Z[g,h] \subset \Q(\mathcal{H}_D)$ is a squarefree polynomial. 

In fact, the Hilbert modular surface $Y_-(D)$ is itself rational if and only if $D = 5, 8, 12, 13, 17$ (see \cite[Theorem~VII.3.4]{vdG}).  Elkies and Kumar also give rational parametrisations $\Q(Y_-(D)) \cong \Q(m,n)$ in these cases, together with the rational maps $(m,n) \mapsto (g,h)$ induced by forgetful morphisms $Y_-(D) \to \mathcal H_D$.

By construction, a $k$-rational point $(z,g,h) \in Y_{-}(D)$ corresponds to genus 2 curve $C/\overbar{k}$ with RM $D$. Indeed, if $\Aut(C) \cong C_2$, then $C$ admits a model over $k$ if and only if the \emph{Mestre obstruction} vanishes, i.e., if the \emph{Mestre conic} associated to the image of $(g,h)$ in $\mathcal{M}_2$ has a $k$-rational point (for more details see~\Cref{sec:mestre-conic}). We write $L_D$ for the Mestre conic associated to the generic point of $\mathcal{H}_D$.

In general, the conic $L_D$ has quite complicated coefficients in $g$ and $h$, and there is no obvious simple criterion for the Mestre obstruction to vanish. An argument of Poonen (see \Cref{prop:pD-is-a-norm}) shows that $L_D$ is isomorphic over $\Q(Y_{-}(D))$ to a conic of the form
\begin{equation}
  \label{eq:pDform}
  \widetilde{{L}}_D \colon X^2 - D Y^2 - q_D Z^2 = 0
\end{equation}
for some rational function $q_D \in \Q(\mathcal{H}_D) \subset \Q(Y_{-}(D))$.  The proof of this result is 
non-constructive, and to actually find such a function $q_D$ (as well as the
requisite transformations) is not at all easy. The main step in the proof of \Cref{thm:models} is to find such transformations, and in particular we have the following theorem.

\begin{thm}
  \label{thm:pDs}
  For $D \in \{5, 8, 12, 13, 17, 21, 24, 28, 29, 33, 37, 44, 53, 61 \}$, the Mestre conic $L_D$ is isomorphic to
$X^2 - D Y^2 - q_D Z^2$ where $q_D \in \Q(g,h)$ is the polynomial of Elkies--Kumar coordinates $(z,g,h)$ given in \cref{tab:qD}.  

Furthermore, for the $5$ fundamental discriminants 
$D \in \{ 5, 8, 12, 13, 17 \}$ such that
$Y_-(D)$ is birational to $\mathbb{A}^2$, the Mestre conic $L_D$ is isomorphic to
$X^2 - D Y^2 - p_D Z^2$ where $p_D \in \Q(m,n)$ is the polynomial in Elkies--Kumar coordinates $(m,n)$ given in
\cref{tab:pD}.
\end{thm}
  
\begingroup
\renewcommand*{\arraystretch}{1.333}
\begin{table}
  \begin{tabular}{r|C{130mm}}
    $D$ & $q_D$                                                                                    \\
    % remarks
    \hline
    5   & $-6(10g+3)(15g+2)$                                                                        \\
    % & factor of $f_D(m,n)$ and $\Delta_L$  \\
    8   & $4g+4h-7$                                                                                \\
    % &  factor of $I_{10}$ and $\Delta_L$   \\
    12  & $-(h - 1) (3 h^{3} + 9 h^{2} - 27 g - 4 h - 8)$                                          \\
    % & factor of $f_D(m,n)$, $I_{10}$, and $\Delta_L$ \\
    13  & $-100 g^{2} + 385 g h - 48 h^{2} + 194 g + 168 h - 108$                                   \\
    % & factor of $f_D(m,n)$ and $\Delta_L$  \\
    17  & $1$                                                                                      \\
    % & $D \equiv 1 \mod 8$                  \\
    21  & $18g^2 - 12gh - 12h^2 - 14$                                                              \\
    % & preserved by extra involution        \\
    24  & $12gh^2 -3g^2   - 2h^2 + 3$                                                              \\
    % & preserved by extra involution        \\
    28  & $- 2 ( 19 g^{2} + 35 h^{2} + 84 h + 28 )$                                                \\
    % & preserved by extra involution        \\
    29  & $- 6 g^{2} - 6 g h + 65 g - 16 h^{2} - 156 h + 4 $                                       \\
    % &                                      \\
    33  & $1$                                                                                      \\
    % & $D \equiv 1 \mod 8$                  \\
    37  & $g^{2} + 15 g h + 20 g - 27 h^{2} + 2 h - 11 $                                           \\
    % &                                      \\
    44  & $( g h + h - 1 ) ( 5 g^{3} h + 9 g^{2} h + 6 g^{2} - 4 g h + 18 g - 8 h + 19 )$          \\
    % &  $g h + h - 1$ is a factor of  $f_D(m,n)$, $I_{10}$, and $\Delta_L$ \\
    53  & $-(25 h^2 + 42 h + 24)g^2 - (h + 1)^2(26 h + 7)g - 11(h + 1)^4$                          \\
    % &\\
    61  & $-3(3h^2 + 7h - 1)g^2 + 2(9h^3 + 12h^2 - 10h - 1)g - 9h^{4} - 3h^{3} + 8h^{2} + 8h - 20$ \\
    % &\\
  \end{tabular}
  \vspace{2mm}
  \caption{Rational functions $q_D \in \Q(\mathcal{H}_D)$ such that the Mestre conic $L_D$ is isomorphic over $\Q(Y_{-}(D))$ to $X^2 - DY^2 - q_DZ^2 = 0$.}
 \label{tab:qD}
\end{table}

\begin{table}
  \begin{tabular}{r|C{130mm}}
    $D$ & $p_D$                                                         \\
    \hline
    5   &$m^2 - 5n^2 - 5$                                               \\
    % & factor of $f_D(m,n)$ and $\Delta_L$  \\
    8   & $-(m+1)$                                                         \\
    % &  factor of $I_{10}$ and $\Delta_L$   \\
    12  & $-27m^2 + n^2 + 27$                                           \\
    % & factor of $f_D(m,n)$, $I_{10}$, and $\Delta_L$ \\
    13  & $1803m^2 - 72mn + n^2 + 3168m - 1440n - 768$                  \\
    % & factor of $f_D(m,n)$ and $\Delta_L$  \\
    17  & 1                                                             \\
    % & $D \equiv 1 \mod 8$                   
  \end{tabular}
  \vspace{2mm}
  \caption{Rational functions $p_D \in \Q(m,n)$ such that the Mestre conic $L_D$ is isomorphic to $X^2 - DY^2 - p_DZ^2 = 0$ when $Y_-(D)$ is rational.} 
 \label{tab:pD}
\end{table}
\endgroup

\begin{rem}
  It is natural to wonder about the significance of the curves that $q_D$ and $p_D$ cut out in $Y_-(D)$.  We make some remarks about this in \Cref{sec:qD}, including a conjecture for $q_{40}$, but this remains quite mysterious to us.
\end{rem}

The first part of \cref{thm:pDs}, together with Mestre's method of constructing a genus 2 curve from a rational point on $L_D$ yields \cref{thm:models}, which we recall gives models in 5 parameters with 2 relations. Using the latter part of this theorem for $D = 5, 8, 12, 13, 17$ provides models in 4 parameters with 1 relation. In fact \cref{thm:pDs} implies that for each $D \leq 17$, the Mestre conic bundle $\mathscr{L}_D \to Y_-(D)$ is a rational threefold. Parametrising $\mathscr{L}_D$ then allows us to give models in 3 parameters with no relations.  

Conversely, the conic bundle $\mathscr{L}_D$ is never (uni-)rational when $D > 17$ since the Hilbert modular surface $Y_{-}(D)$ is not rational; see~\cite[Theorem~VII.3.4]{vdG}. Note that, since $q_D$ is a rational function on $\mathcal{H}_D$, equation \eqref{eq:pDform} defines a conic over $\Q(\mathcal{H}_D)$. This conic spreads out to a conic bundle $\mathscr{L}_D^{\mathcal{H}} \to \mathcal{H}_D$ which, \textit{a priori}, may be rational. In such cases it is possible to give a $4$-parameter family $y^2 = f_D(a,b,c,z;x)$ subject to a single relation $z^2 = \Lambda(a,b,c)$ for some polynomial $\Lambda(a,b,c) \in \Z[a,b,c]$.

\begin{question}
  Does there exist a fundamental discriminant $D$ such that ${\mathscr{L}}^{\mathcal H}_D$ is not geometrically (uni-)rational but $\mathcal{H}_D$ is a rational surface?
\end{question}

When $D = 21$, $28$, $29$, $33$, $37$, $44$, $53$, and $61$ we exhibit explicit parametrisations of the threefolds ${\mathscr{L}}^{\mathcal H}_D$, which may be found in~\cite{electronic}. Moreover, we give a parametrisation of $\mathscr{L}_{24}^{\mathcal{H}}$ over $\Q(\sqrt{-2})$.

Note that ${\mathscr{L}}^{\mathcal H}_D$ is not birational to the Mestre conic bundle over $\mathcal{H}_D$. Rather, \cref{prop:pD-is-a-norm} implies that the Mestre conic over $\Q(\mathcal{H}_D)$ is isomorphic to a conic of the form $X^2 - D \lambda_D Y^2 - q_D Z^2 = 0$.  It is not clear if the Mestre conic bundle over $\mathcal{H}_D$ is ever rational.

As mentioned above, when $D=17$ we may remove one more parameter without losing genericity. This is because the Mestre obstruction vanishes whenever $D \equiv 1 \pmod 8$.

\begin{thm}
  \label{thm:qD=1}
  If $D \equiv 1 \pmod 8$ is a positive fundamental discriminant,
  then the Mestre conic $L_D$ is isomorphic over $\Q(Y_{-}(D))$ to a conic of the form $X^2 - D Y^2 - Z^2$, i.e., we can take $q_D = 1$.
\end{thm}

In \Cref{sec:theoretical-results} we prove \cref{thm:qD=1} by studying the RM action on the $2$-torsion on a Jacobian with RM $D$. When $D \equiv 1 \pmod 8$, the prime $2$ splits in $\Q(\sqrt D)$, and using this observation we prove that a genus 2 curve with such RM has a Weierstrass model of the form $y^2 = \poly(x)\polytwo(x)$ for some cubic polynomials $\poly(x), \polytwo(x) \in k[x]$ (see \cref{lemma:1-mod-8-factor}).  These facts also allow us to get a relatively nice model in the case $D = 33$.

\cref{thm:qD=1} implies that whenever $D \equiv 1 \pmod 8$, the image of the set of rational points on $Y_-(D)$ under the double covering map $Y_-(D) \to \mathcal H_D$ generically parametr{\iz}es genus 2 curves with RM $D$ over $k$ (up to twist).  However, \cref{thm:qD=1} by itself is not sufficient to construct generic models.

\subsection{An outline of the proof of \texorpdfstring{\Cref{thm:models}}{Theorem 1.1}: Simplifying the Mestre conic}
To prove Theorems~\ref{thm:models} and \ref{thm:pDs} we construct transformations which put the Mestre conic $L_D$ into one of the form \eqref{eq:pDform}.
In \cite{CM:RM5} the first and third authors carried this out using naive methods when $D=5$. However, that approach is not practical for $D > 5$.

The main idea is to employ \emph{minim{\iz}ation}, similar to Tate's algorithm for elliptic curves. We regard a conic $L/\Q(t_1,t_2)$ as the generic fibre of a conic defined over $\Z[t_1,t_2]$. By repeatedly blowing-up singularities of the latter we are often able to \emph{minim{\iz}e} $L$, that is, find an isomorphic conic whose discriminant has minimal degree (viewed as an element of $\Z[t_1,t_2]$).  However $\Z[t_1, t_2]$ is not a PID so, unlike when $L$ is defined over $\Q$, this process turns out to be quite delicate as:
\begin{enumerate}
\item
  most blow-ups introduce other bad primes, which may be worse, and
\item
  this process is extremely sensitive to the order in which these blow-ups are made.
\end{enumerate}
For most $D < 30$, with a lot of patience and various tricks, one can carry out this process ``by hand'' 
(i.e., human-directed calculations in \texttt{Magma}~\cite{MAGMA}) to find $q_D$, however the
complexity grows quickly with $D$. 

In \Cref{sec:algorithms} we develop and implement an algorithm which automates this minim{\iz}ation
process.  By running this algorithm we find the transformations necessary to prove \cref{thm:pDs}. Applying these transformations we find the corresponding generic models (see~\Cref{sec:proof-main-results}).  
Our \texttt{Magma} implementation has been made available through the GitHub repository~\cite{electronic}.

The minim{\iz}ation approach which we employ is likely well known to experts, but the automation aspect seems to
be novel (and can be adapted to much more general situations than conics
defined over 2-variable function fields).  
To be clear, we are not able to give an effective algorithm to minim{\iz}e a conic
in the sense of provably terminating with a solution.  Rather, we give an
algorithm to search through a tree of sequences of blow-ups, which
involves carefully scoring and pruning paths. This is necessary since the whole search tree is massive
(it is infinite) and individual blow-up calculations can get very slow.
Runtimes for our algorithm can be found in \cref{tab:runtimes}.
It is not clear that increasing computational resources would allow us to treat more discriminants.

\begin{rem}
  The preliminary step in our approach has applications to reconstructing general genus 2 curves (i.e., without an RM condition) from their moduli. Mestre's original conic is given in a simple form in terms of Clebsch invariants, however, in practice, one often begins with Igusa--Clebsch invariants and writes the Mestre conic in terms of Igusa--Clebsch invariants.  The resulting coefficients can be quite large even when the Igusa--Clebsch invariants are small. In \Cref{sec:ICmestre} we present simplified forms of Mestre's conic in terms of Igusa--Clebsch invariants.  In addition to simplifying our minim{\iz}ation process, this can also be used to construct genus 2 curves from Igusa--Clebsch invariants with smaller coefficients than the standard procedure.
\end{rem}

\subsection{Further remarks}

First, we say a little more about previous works on families of genus 2 curves with RM.
In \cite{mestre:families}, Mestre constructed 2-parameter families with RM 5 and RM 8. These families, however, do not contain all twist classes over $\Q$.  Brumer (see \cite{brumer:rank,hashimoto}) constructed a 3-parameter family with RM 5 --- it is not known if Brumer's family covers all twist classes over $\Q$.  
Bending \cite{Bending_thesis,Bending_paper} gave a versal family with RM 8 determined (up to quadratic twist) by three parameters $A,P,Q \in k$ (that is, every genus $2$ curve with RM 8 arises in this way, but distinct parameters $(A,P,Q)$ may correspond to $\bar k$-isomorphic curves).  
These constructions, however, are quite specific to discriminants 5 and 8.
Moreover, it is not clear from these constructions which parameters yield
 isomorphic curves, either over $k$ or $\overbar k$.

To our knowledge, aside from being generic models, ours are the first explicit models of any families of genus 2 curves with RM $D$ beyond the cases of $D = 5, 8$, and the recent work of \cite{Bruin-Flynn-Shnidman} for $D=12$ with full $\sqrt 3$-level structure.
See \cref{sec:relat-prev-work} for comparisons between our models and existing families.

In \cite{Elkies-Kumar}, for many fundamental discriminants $D < 100$, Elkies and Kumar identified one or more curves on $Y_-(D)$ with no Brauer obstruction, which typically correspond to 1-parameter families of genus 2 curves over $\Q$ with RM $D$.  (Note that modular curves on $Y_-(D)$ need not correspond to families of genus 2 curves where the RM is defined over $\Q$.)  However, no explicit models were given.

One application of such explicit models is to estimate counts of genus 2 curves with RM $D$ by discriminant or conductor, which translates into estimates of counts of weight 2 newforms with rationality field $\Q(\sqrt D)$.  In \cite{CM:conjectures}, we used existing models to prove lower bounds for such counts for $D = 5, 8$.  Our models here should have similar applications to other discriminants $D$.

\vspace{-1mm}
\subsection*{Acknowledgements}
We thank Noam Elkies, Tom Fisher, and Jef Laga for helpful discussions. We are grateful to Bjorn Poonen for kindly explaining \cref{prop:pD-is-a-norm} and its proof to us, and allowing us to include it here.

The algorithms in the paper have been implemented in \texttt{Magma} and our implementation has been made publicly available through the GitHub repository \cite{electronic}.

AC was supported by the Simons Foundation Collaboration Grant 550031.
SF was supported by the Woolf Fisher and Cambridge Trusts and by Royal Society Research Fellows Enhanced Research Expenses RF{$\backslash$}ERE{$\backslash$}210398.
KM was supported by the Japan Society for the 
Promotion of Science (Invitational Fellowship L22540), and the
Osaka Central Advanced Mathematical Institute (MEXT Joint Usage/Research Center on Mathematics and Theoretical Physics JPMXP0619217849).

This work was also supported by the National Science Foundation under Grant No. DMS-1929284 while the authors were in residence at the Institute for Computational and Experimental Research in Mathematics in Providence, RI, during Collaborate@ICERM ``Rational Genus 2 Curves with Real Multiplication.''  
Calculations were performed using computational resources and services at the Center for Computation and Visual{\iz}ation, Brown University.
We thank ICERM and Brown University for their hospitality, and providing access to their OSCAR Supercomputer.
\vspace{-1mm}

\section{Background}
\subsection{Real multiplication and Hilbert modular surfaces}
\label{sec:real-mult-hilb}
Let $k$ be a field of characteristic $0$, and let $C/k$ be a genus $2$ curve. We equip the Jacobian $J = \Jac(C)$ with the canonical principal polarisation and for each endomorphism $\psi \in \End_{\overbar k}(J)$ we write $\psi^\dagger$ for its image under the Rosati involution, that is $\psi^\dagger = \zeta^{-1} \circ \hat{\psi} \circ \zeta$ where $\zeta$ is the principal polarisation on $J$ and $\hat{\psi}$ is the dual of $\psi$. Denote by $\End_{\overbar k}^\dagger(J) \subset \End_{\overbar k}(J)$ the subalgebra of endomorphisms fixed by the Rosati involution, i.e., for which $\psi = \psi^\dagger$.

Let $D$ be a fundamental discriminant. If $\End_{\overbar k}^\dagger(J)$ contains the quadratic order $\calO_D$ of discriminant $D$, we say $J$ (and $C$) 
has real multiplication (RM) by $\calO_D$, and abbreviate this as RM $D$.

Let $\mathcal{M}_2$ and $\mathcal{A}_2$ denote the (coarse) moduli spaces of genus $2$ curves and principally polarised abelian surfaces, respectively. We define for the \emph{Hilbert modular surface $Y_-(D)$ of discriminant $D$} to be the coarse moduli space parametr{\iz}ing pairs $(J/k, \iota)$ where $J/k$ is a genus $2$ Jacobian, and $\iota \colon \calO_D \hookrightarrow \End^\dagger_k(J)$ is an injective ring homomorphism.
(This is birational to the definition in terms of principally polar{\iz}ed abelian surfaces used in 
\cite{vdG} and \cite{Elkies-Kumar}.)

There is a natural forgetful map $Y_{-}(D) \to \mathcal{A}_2$ given functorially by forgetting $\iota$. The (Zariski closure of) the image of this map is known as the \emph{Humbert surface} (of discriminant $D$) and we denote it by $\mathcal{H}_D$.

\begin{rem}
  \label{rem:genericity}
  Let $\mathscr{L}/k$ be a geometrically integral variety. Since a $k$-point on $Y_-(D)$ may not correspond to a genus $2$ curve defined over $k$ we adopt the following convention: A curve $\mathcal{C}/\Q(\mathscr{L})$ is said to be a \emph{generic curve with RM $D$} if there exists a Zariski dense set $U \subset Y_{-}(D)$ such that if $P \in U(k)$ is the moduli of a genus $2$ curve $C/k$, then $P$ lifts to a point $P' \in \mathscr{L}(k)$ such that the specialisation of $\mathcal{C}$ at $P'$ is $\bar k$-isomorphic to $C$.
\end{rem}

\subsection{The Mestre conic and cubic}
\label{sec:mestre-conic}
Let $\mathrm{M}_3(R)$ for the $R$-algebra of $3 \times 3$ matrices with entries in $R$.
If $L/R$ is a conic given by the vanishing of a homogeneous degree $2$ polynomial $Q(X,Y,Z) \in R[X,Y,Z]$ we define the {Gram matrix} of $L$ to be the symmetric matrix $A \in \mathrm{M}_3(R)$ such that
\begin{equation*}
  Q(X,Y,Z) = \mathbf{x}^{\mathrm{T}} A  \mathbf{x}.
\end{equation*}
The {discriminant} of the conic $L$ (and the polynomial $Q(X,Y,Z)$) is defined to be the determinant $\Delta(L) = \Delta(Q) = \det A$.

Let $\mathcal{S} \subset \operatorname{Sym}^6 (\mathbb{P}^1)$ be the moduli space parametrising $6$ \emph{distinct} (unordered) points in $\PP^1$, equipped with the natural action of $\Aut(\PP^1) \cong \PGL_2$. To each point $\boldsymbol{a} = \{a_1, ..., a_6\} \in \mathcal{S}(k)$ we may associate the genus $2$ curve $C_{\boldsymbol{a}}/k$ given by the Weierstrass equation
\begin{equation*}
  C_{\boldsymbol{a}} : y^2 = \prod_{i=1}^6 (x - a_i) .
\end{equation*}

The isomorphism class of the curve $C_{\boldsymbol{a}}$ is invariant under the action of $\PGL_2$ on $\mathcal{S}$. Since all genus $2$ curves are hyperelliptic we therefore have a $\Q$-isomorphism $\mathcal{M}_2 \cong \mathcal{S} \!\sslash\! \PGL_2$. There exist $\PGL_2 \times S_6$-invariants in $\Z[a_1,...,a_6]$, known as the \emph{Igusa--Clebsch invariants}, of degrees $2$, $4$, $6$, and $10$ respectively such that the induced morphism $\mathcal{M}_2 \to \PP(1,2,3,5)$ is $\Q$-birational, and is a $\Q$-isomorphism onto its image (see e.g.,~\cite{IgusaM2}). We identify $\mathcal{M}_2$ with its image in $\PP(1,2,3,5)$.

Consider a $k$-rational point $P = [I_2:I_4:I_6:I_{10}] \in \mathcal{M}_2 \subset \PP(1,2,3,5)$. By construction we may associate a genus $2$ curve $C/k$ to $P$ if and only if the fibre of the morphism $\mathcal{S} \to \mathcal{M}_2$ above $P$ has a $k$-rational point. Generically, this fibre is a $\PGL_2$-torsor, hence for general $P$ it is isomorphic to a conic $L(P)/k$. Mestre~\cite{mestre} proved that if $P$ corresponds to a genus $2$ curve $C/\overbar{k}$ with $\Aut_{\overbar{k}}(C) \cong C_2$ then $L(P)$ is isomorphic to the conic with (symmetric) Gram matrix $(A_{ij})_{i,j=1}^3$ whose upper triangular entries are
\begingroup
\allowdisplaybreaks
\begin{align*}
A_{1,1} &= \frac{-3 I_{2}^{3} - 140 I_{2} I_{4} + 800 I_{6}}{2^6 \cdot 3^4 \cdot 5^6}, \\
A_{1,2} &= \frac{9 I_{2}^{4} + 560 I_{2}^{2} I_{4} + 1600 I_{4}^{2} - 3000 I_{2} I_{6}}{2^7 \cdot 3^7 \cdot 5^8},\\
A_{1,3} &= \frac{-9 I_{2}^{5} - 700 I_{2}^{3} I_{4} + 12400 I_{2} I_{4}^{2} + 3600 I_{2}^{2} I_{6} - 48000 I_{4} I_{6} - 10800000 I_{10}}{2^8 \cdot 3^9 \cdot 5^{10}}\\
A_{2,2} &= A_{1,3}, \\
A_{2,3} &= \frac{3 I_{2}^{6} + 280 I_{2}^{4} I_{4} + 6000 I_{2}^{2} I_{4}^{2} - 1400 I_{2}^{3} I_{6} + 8000 I_{4}^{3} - 52000 I_{2} I_{4} I_{6} + 120000 I_{6}^{2}}{2^9 \cdot 3^{10} \cdot 5^{12}},\\
A_{3,3} &= \frac{-9 I_{2}^{7} - 980 I_{2}^{5} I_{4} - 12800 I_{2}^{3} I_{4}^{2} + 4800 I_{2}^{4} I_{6} + 154000 I_{2} I_{4}^{3} + 162000 I_{2}^{2} I_{4} I_{6}}{2^{10} \cdot 3^{13} \cdot 5^{14}}\\ &\hspace{0.5cm}-\frac{480000 I_{4}^{2} I_{6} + 450000 I_{2} I_{6}^{2} + 8100000 I_{2}^{2} I_{10} + 162000000 I_{4} I_{10}}{2^{10} \cdot 3^{13} \cdot 5^{14}}.
\end{align*}
\endgroup

Mestre also constructs an explicit cubic curve $M(P) \subset \PP^2$ which is defined over $k$. The curve $M(P)$ has the following property.

\begin{prop}[{Mestre~\cite[\S1.5]{mestre}}]
  \label{prop:conic-cubic}
  Let $K/k$ be a field. If there exists a $K$-isomorphism $\psi \colon L(P) \cong \PP^1$ (i.e., if $L(P)$ has a $K$-point), then the double cover of $\PP^1$ ramified over $\psi(M(P) \cap L(P))$ is a genus $2$ curve $C/K$ with moduli $P \in \mathcal{M}_2(K)$.   
\end{prop}

We do not reproduce the equations for $M(P)$ here, but they may be found in \cite{mestre} and in the code attached to this paper~\cite{electronic}.

\section{The Mestre conic associated to a Hilbert modular surface}
\label{sec:theoretical-results}

Let $Y_{-}(D)$ be the Hilbert modular surface of fundamental discriminant $D$ and let $\mathcal{H}_D \subset \mathcal{A}_2$ be the corresponding Humbert surface, as in \Cref{sec:real-mult-hilb}.
Let $P \in \mathcal{H}_D(k) \cap \mathcal{M}_2(k)$ be a $k$-rational point and consider a genus $2$ curve $C/\overbar{k}$ corresponding to $P$. We say that $P$ is \emph{suitably generic} if the endomorphism ring of the Jacobian of $C$ is commutative and $\Aut(C) \cong \{ \pm 1\}$. Similarly, we say that a point $P \in Y_{-}(D)(k)$ is suitably generic if its image in $\mathcal{H}_D(k)$ is suitably generic. 

The following result and argument was kindly explained to us by Bjorn Poonen.

\begin{prop}
  \label{prop:pD-is-a-norm}
  Let $k$ be a field of characteristic $0$ and let $P \in \mathcal{H}_D(k)$ be a suitably generic $k$-rational point. Let $\lambda(P) \in k$ be such that $k(Q) = k(\sqrt{\lambda(P)})$ where $Q \in Y_{-}(D)(\overbar{k})$ is a preimage of $P$. Then the Mestre conic $L_D(P) \subset \PP^2$ associated to $P$ is $k$-isomorphic to a conic  of the form
  \begin{equation*}
    X^2 - D \lambda(P) Y^2 - q_{{D}}(P) Z^2
  \end{equation*}
  for some $q_D(P) \in k$.
\end{prop}

\begin{rem}
  If $P$ is the generic point of $\mathcal{H}_D$, then $\lambda_D \colonequals \lambda(P) \in \Q(\mathcal{H}_D)^\times$ is the function (unique up to multiplication by a square) such that $Y_{-}(D)$ is birational to the surface given by the vanishing of $z^2 - \lambda_D$ in $\mathbb{A}^1 \times \mathcal{H}_D$.
\end{rem}

\begin{proof}
  For simplicity write $\lambda = \lambda(P)$. By \cite[Theorem~5.5.3]{VoightQuaternions} it suffices to show that there exists a $k(\sqrt{\lambda D})$-rational point on $L_D(P)$. 
  
  Recall that we write $\mathcal{S}$ for the moduli space parametrising $6$ distinct points in $\PP^1$ and that we have an isomorphism $\mathcal{M}_2 \cong \mathcal{S} \sslash \PGL_2$. The fibre of the quotient morphism $\mathcal{S} \to \mathcal{M}_2$ above $P$ is a $\PGL_2$-torsor, and by construction defines the same class in $H^1(k, \PGL_2)$ as $L_D(P)$.

  Note that we may equivalently view $\mathcal{S}$ as the (coarse) moduli space parametrising pairs $(C, \phi)$ where $C$ is a genus $2$ curve, and $\phi \colon C \to \PP^1$ is a morphism of degree $2$. Define $\mathcal{N}$ to be the (coarse) moduli space parametrising triples $(C, \alpha, \beta)$ where $C$ is a genus $2$ curve and $\alpha,\beta \in |\mathcal{K}_C|$ are elements of the linear system of a canonical divisor $\mathcal{K}_C$ on $C$. Note that if $\Aut(C) \cong \{\pm 1\}$ then the degree $2$ morphism $\phi \colon C \to \PP^1$ is unique up to composition with an element of $\Aut(\PP^1)$, and therefore  $\phi^*(0), \phi^*(\infty) \in |\mathcal{K}_C|$.

  Let $\mathcal{M}_2'$ denote the Zariski open subvariety of $\mathcal{M}_2$ where $\Aut(C) \cong \{\pm 1\}$ and define $\mathcal{S}' = \mathcal{S} \times_{\mathcal{M}_2} \mathcal{M}_2'$ and $\mathcal{N}' = \mathcal{N} \times_{\mathcal{M}_2} \mathcal{M}_2'$. We have a commutative diagram 
  \begin{equation*}
    \begin{tikzcd}
      \mathcal{S}' \\
      \mathcal{N}' & \mathcal{M}_2'
      \arrow[from=1-1, to=2-1]
      \arrow[from=2-1, to=2-2]
      \arrow[from=1-1, to=2-2]
    \end{tikzcd}
  \end{equation*}
  where the morphism on the left is given by $(C, \phi) \mapsto (C, \phi^{*}(0), \phi^{*}(\infty))$. Hence $\mathcal{S}' \to \mathcal{N}'$ is a $\mathbb{G}_m$-torsor since $\mathbb{G}_m$ is isomorphic to the subgroup of $\Aut(\PP^1)$ fixing the points $0$ and $\infty$.

  By Hilbert's Theorem 90 a $\mathbb{G}_m$-torsor is trivial, so it suffices to show that, over $k(\sqrt{\lambda D})$, there exists a section $P \to \mathcal{N}'$ of the morphism $\mathcal{S}' \to \mathcal{N}'$.

  Choose a finite extension $K/k$ such that there exists a genus $2$ curve $C/K$ corresponding to $P$. Let $J = \Jac(C)$ be the Jacobian of $C$. By construction, there exists an element $\tau \in \End_{\overbar{k}}(J)$ such that $\tau^2 = D$. Since $J$ has abelian geometric endomorphism algebra (by assumption) $\tau$ is defined over $K(\sqrt{\lambda})$ by \cite[Proposition~2.1]{CM:RM5}.

  Note that the action of $\tau$ on $J$ induces an action on $H^0(J, \Omega_J) \cong H^0(C, \Omega_C)$ which, after base-changing to $K(\sqrt{\lambda})$, we may assume is given by the matrix $\bigl( \begin{smallmatrix}0 & 1\\ D & 0\end{smallmatrix}\bigr)$. Let $\omega_1, \omega_2 \in H^0(C, \Omega_C)$ be eigenvectors contained in the span of the $\sqrt{D}$ and $-\sqrt{D}$-eigenspaces of this action respectively.

  It is clear that $\omega_1$ and $\omega_2$ are fixed by the action of $\Gal(\overbar{k}/ K(\sqrt{\lambda}, \sqrt{D}))$. Let
  \begin{equation*}
    \chi_\lambda, \chi_D \colon \Gal(\overbar{k} / K) \to \{ \pm 1 \}
  \end{equation*}
  denote the quadratic characters associated to $K(\sqrt{\lambda})$ and $K(\sqrt{D})$ respectively. Then for each $i = 1,2$ the $1$-form $\omega_i$ is fixed by the action of $\sigma \in \Gal(\overbar{k}/K)$ if and only if $\chi_\lambda(\sigma) \chi_D(\sigma) = 1$. In particular both $\omega_1$ and $\omega_2$ are defined over $K(\sqrt{\lambda D})$.

  By abuse of notation let $\omega_1, \omega_2 \in |\mathcal{K}_C|$ denote the degree $2$ divisors cut out by the $1$-forms $\omega_1$ and $\omega_2$ respectively. By the above discussion the point $(C, \omega_1, \omega_2) \in \mathcal{N}(\overbar{k})$ is $K(\sqrt{\lambda D})$-rational, and since $P$ is $k$-rational we in fact have $k(\sqrt{\lambda D}) = \bigcap_{K} K(\sqrt{\lambda D})$, where we range over all possible choices of $K/k$. Since $K$ was arbitrary $(C, \omega_1, \omega_2)$ is $k(\sqrt{\lambda D})$-rational as required.
\end{proof}

The following proposition shows that the Mestre obstruction vanishes identically on the Hilbert modular surface $Y_{-}(D)$ for any $D \equiv 1 \pmod{8}$. \Cref{thm:qD=1} follows by taking the point $P \in Y_{-}(D)$ in \Cref{prop:1-mod-8} to be the generic point. That is, when $D \equiv 1 \pmod{8}$ we may take $q_D(P) = 1$ in \Cref{prop:pD-is-a-norm}.

\begin{prop}
  \label{prop:1-mod-8}
  Let $k$ be a field of characteristic coprime to $2D$ and let $P \in Y_{-}(D)(k)$ be a suitably generic $k$-rational point. The Mestre conic $L_D(P) \subset \PP^2$ associated to $P$ has a $k$-rational point whenever $D \equiv 1 \pmod{8}$.
\end{prop}

\begin{proof}
  If $k$ is finite, then any conic has a $k$-rational point, so we may assume without loss of generality that $k$ is infinite. We first need the following lemmas:

  \begin{lemma}
    \label{lemma:orthog-decomp} Let $D > 0$ be any fundamental discriminant. Suppose that $p$ is a prime number which splits in $\mathcal{O}_D$ as a product of prime ideals $(p) = \mathfrak{p}\bar{\mathfrak{p}}$.
    Let $k$ be a field of characteristic coprime to $pD$ and let $J/k$ be a principally polarised abelian surface with RM $D$ defined over $k$. Denote by $e_p \colon J[p] \times J[p] \to \mu_p$ the $p$-Weil pairing induced by the principal polarisation on $J$. Then we have a direct sum decomposition of Galois modules $J[p] = J[\mathfrak{p}] \oplus J[\bar{\mathfrak{p}}]$ and moreover $e_p(x,y) = 1$ for each pair $x \in J[\mathfrak{p}]$ and $y \in J[\bar{\mathfrak{p}}]$.
  \end{lemma}

  \begin{proof}
    The decompostion $J[p] = J[\mathfrak{p}] \oplus J[\bar{\mathfrak{p}}]$ is an isomorphism of Galois modules since the action of $\mathcal{O}_D$ is defined over $k$.
    
    We claim that for each $y \in J[\bar{\mathfrak{p}}]$ there exists $y' \in J[p]$ and $\eta \in \mathfrak{p}$ such that $y = \eta y'$. If there exists $\eta \in \mathfrak{p}$ such that $\eta J[p] = J[\bar{\mathfrak{p}}]$ then this is clear, so suppose that for each $a \in \mathfrak{p}$ we have $a J[p] \subsetneq J[\bar{\mathfrak{p}}]$. Since $J[\mathfrak{p}]$ is an $\F_p$-module of rank $2$ there exist $a, b \in \mathfrak{p}$ such that $\ker a \neq \ker b$. In particular $a J[p]  \neq b J[p]$, so choosing $z \in J[p] \setminus (\ker a \cup \ker b)$ we see that $\{a z, b z \}$ is an $\F_p$-basis for $J[\bar{\mathfrak{p}}]$ and therefore there exist $m,n \in \Z$ such that $y = (n a + m b) z$.

    The $\mathcal{O}_D$-action on $J$ is invariant under the Rosati involution by assumption, and therefore we have $e_p(a x, y) = e_p(x, a y)$ for any $a \in \mathcal{O}_D$ and $x,y \in J[p]$. But $x \in \ker \eta$ since $\eta \in \mathfrak{p}$, and therefore $e_p(x,y) = e_p(x,\eta y') = e_p(\eta x, y') = 1$, as required.
  \end{proof}
  
  \begin{lemma}
    \label{lemma:1-mod-8-factor}
    Let $k$ be a field of characteristic coprime to $2D$, $C/k$ a curve of genus $2$, $D \equiv 1 \pmod{8}$, and suppose that $J = \Jac(C)$ has RM $D$ defined over $k$. Let $C : y^2 = f(x)$ be a Weierstrass equation for $C$, where $f(x) \in k[x]$ is a sextic polynomial. Then there exists a factorisation $f(x) = \poly(x) \polytwo(x)$ as a product of two cubic polynomials over $k$.
  \end{lemma}

  \begin{proof}
    Since $D \equiv 1 \pmod{8}$ the ideal $(2) \subset \mathcal{O}_D$ splits as a product $(2) = \mathfrak{p} \bar{\mathfrak{p}}$. We equip $J$ with the canonical principal polar{\iz}ation and let $e_2 \colon J[2] \times J[2] \to \mu_2$ denote the induced $2$-Weil pairing on $J$. 

    Let $x,y \in J[\mathfrak{p}] \setminus \{O\}$ be distinct elements. We claim that $e_2(x,y) \neq 1$. Indeed, if $e_2(x,y) = 1$ then since $J[2]$ is spanned by $x$, $y$, and $J[\bar{\mathfrak{p}}]$, by \Cref{lemma:orthog-decomp} the $2$-Weil pairing $e_2 \colon J[2] \times J[2] \to \mu_2$ would be trivial, which is not the case.
    
    Following \cite[Section~4]{Bruin-Doerksen} let $x_1,...,x_6 \in \overbar{k}$ denote the roots of $f(x)$, and let $T_{i,j} \in \Pic^0(C/\overbar{k})$ denote the divisor class $T_{i,j} = T_{j,i} = [(x_i,0) - (x_j,0)]$. By \cite[Lemma~8.1.4]{Smith_Thesis} we have $e_2(T_{i,j}, T_{s,t}) = (-1)^{\#\{i,j,s,t\}}$. In particular, since for each distinct $x,y \in J[\mathfrak{p}] \setminus \{O\}$ we have $e_2(x,y) = -1$, it may be assumed without loss of generality that $J[\mathfrak{p}] = \{O, T_{1,2}, T_{2,3}, T_{1,3}\}$. But the $\Gal(\overbar{k}/k)$-action on $J(\overbar{k})$ stabilises $J[\mathfrak{p}]$, so in particular $\Gal(\overbar{k}/k)$ also stabilises $\{x_1, x_2, x_3\}$, and therefore $\poly(x) = (x - x_1)(x - x_2)(x - x_3)$ is contained in $k[x]$. But then setting $\polytwo(x) = f(x)/\poly(x)$ we obtain the factor{\iz}ation $f(x) = \poly(x)\polytwo(x)$, as required.
  \end{proof}

  Let $L = L_D(P)$ and $M = M_D(P)$ denote the Mestre conic and cubic associated to $P$ respectively. Since $k$ is infinite there exists a separable quadratic extension $K/k$ such that $L$ obtains a point over $K$, and $K \cap k(L \cap M) = k$. By construction the Mestre obstruction for $P$ vanishes over $K$, and there exists a genus $2$ curve $C/K$ corresponding to $P$. The geometric endomorphism ring $\End_{\overbar k}(\Jac(C))$ is abelian since $P$ is assumed to be suitably generic. It follows from \cite[Proposition~2.1]{CM:RM5} that $\Jac(C)$ has $\mathcal{O}_D$-multiplication over $K$, and therefore if $y^2 = f(x)$ is a Weierstrass equation for $C$. By \Cref{lemma:1-mod-8-factor} the polynomial $f(x)$ factors over $K$ as a product of cubics $f(x) = \poly(x) \polytwo(x)$.

  Let $S \subset \mathbb{A}^1$ denote the $K$-variety cut out by the vanishing of the polynomial $f$. By construction we have a $K$-isomorphism of varieties $L \cap M \cong S$. In particular, over $K$, we have a decomposition $L \cap M = S_1 \amalg S_2 $ where $S_1$ and $S_2$ are degree 3 $K$-rational divisors on $L$ (isomorphic as $K$-varieties to the vanishing sets of $\poly(x)$ and $\polytwo(x)$ respectively).

  Since $K \cap k(L \cap M) = k$ the $\Gal(K/k)$-action on $L \cap M$ must stabilise $S_1$ and $S_2$. In particular both $S_1$ and $S_2$ are defined over $k$. But if $\mathcal{K}_{L}$ is a canonical divisor on $L$, then $S_1 - \mathcal{K}_{L}$ is a $k$-rational degree 1 divisor on $L$, and is therefore linearly equivalent to a $k$-rational point on $L$ (since $L$ has genus $0$).  
\end{proof}

%%%%%%%%%%%%%%%%%%%%%%%%%%%%%%%%
%%%%%%%%%%%%%%%%%%%%%%%%%%%%%%%%
%
%  SIMPLIFICATIONS
%
%%%%%%%%%%%%%%%%%%%%%%%%%%%%%%%%
%%%%%%%%%%%%%%%%%%%%%%%%%%%%%%%%

\section{Simplified models for the generic Mestre conic}
\label{sec:simplifications}

Mestre's conic is quite simple in terms of Clebsch invariants, however the
Clebsch invariants are quite complicated rational functions on the Elkies--Kumar model for the Humbert surface. In this section we present two simplifications
of the Mestre conic, firstly in terms of Igusa--Clebsch invariants (see \Cref{sec:ICmestre}), and secondly in terms of
certain related quantities defined by Elkies--Kumar~\cite{Elkies-Kumar} for Hilbert modular surfaces (see \Cref{sec:conic-EK}). In each case, we record the cubic form satisfying the analogue of \Cref{prop:conic-cubic} in the electronic data~\cite{electronic}.

We remark that the first simplification is useful for reconstructing
genus 2 curves over number fields, say, from Igusa or Igusa--Clebsch invariants, as it tends to give genus 2 curves with much smaller coefficients
than using Mestre's construction \cite{mestre} directly. 

\subsection{Mestre conics for Igusa--Clebsch invariants} \label{sec:ICmestre}
Let $T^{(1)}$ be the Gram matrix for the Mestre conic $Q_1$, viewed over $R = k[I_2, I_4, I_6, I_{10}]$.
The upper triangular coefficients of $T^{(1)}$ are given in \cref{sec:mestre-conic}.

Let $e_i$ is the $i$-th standard basis vector of $R^3$ for $i = 1, 2, 3$.
Viewing each of the monomials $I_2, I_4, I_6, I_{10}$ as degree 1 over $k$, we see that $Q_1(e_1), Q_1(e_2), Q_1(e_3)$ have degrees 3, 5, 7.  We will simplify the Mestre conic in terms of Igusa--Clebsch invariants by making some change of bases to lower the degree of the coefficients.

Let $v_1 = I_2 e_2 + 450e_3$ and $v_2 = I_2 e_1 + 450e_2$.
Then $Q_1(v_1)$ has degree 5 and $Q_1(v_2)$ has degree 3.
Let $T^{(2)}$ be the Gram matrix for $Q_1$ with respect to the basis
$\{ e_1, e_2, v_1 \}$.  Let $T^{(3)}$ be the Gram matrix for $Q_2$
with respect to $\{ 9000e_1, 1350v_2, 607500e_3 \}$. 
Now the degrees of the diagonal terms of $T^{(3)}$ are 3, 3, and 5, with the degree 5 entry of $T^{(3)}$ being 
\begin{align*}
T^{(3)}_{3,3} = 267 I_{2}^{3} I_{4}^{2} + 1515 I_{2} I_{4}^{3} - 1485 I_{2}^{2} I_{4} I_{6} - 3600 I_{4}^{2} I_{6}&\\ +\, 2025 I_{2} I_{6}^{2} - 141750 I_{2}^{2} I_{10} - 1215000 I_{4} I_{10}&.
\end{align*}

We can do one more simplification to get rid of the $I_2^2I_{10} $ and 
$I_2^3 I_4^2$ terms from the lower right entry of $T^{(3)}$.  Let $T^{(4)}$ be
the Gram matrix for $Q_3$ with respect to the basis
$\{ e_1, e_2, \frac{v_3}{10} \}$ where $v_1 = 7I_4 e_1 + e_2 I_2 - 2e_3$.
The upper triangular coefficients of $T^{(4)}$ are:
\begin{align*}
T^{(4)}_{1,1} &= -3 I_{2}^{3} - 140 I_{2} I_{4} + 800 I_{6}
\\
T^{(4)}_{1,2} &= 7 I_{2}^{2} I_{4} + 80 I_{4}^{2} - 30 I_{2} I_{6}
\\
T^{(4)}_{1,3} &= -230 I_{2} I_{4}^{2} - 9 I_{2}^{2} I_{6} + 1040 I_{4} I_{6} + 108000 I_{10}
 \\
T^{(4)}_{2,2} &= 117 I_{2} I_{4}^{2} - 360 I_{4} I_{6} - 81000 I_{10}
 \\
T^{(4)}_{2,3} &= -50 I_{2}^{2} I_{4}^{2} + 20 I_{4}^{3} + 321 I_{2} I_{4} I_{6} - 540 I_{6}^{2} + 24300 I_{2} I_{10}
 \\
T^{(4)}_{3,3} &= -200 I_{2} I_{4}^{3} + 920 I_{4}^{2} I_{6} - 27 I_{2} I_{6}^{2} + 102600 I_{4} I_{10}
\end{align*}

We will call this transformed Mestre conic the \emph{IC-simplified Mestre
conic}.

\subsection{Mestre conics for Elkies--Kumar models}
\label{sec:conic-EK}
The Elkies--Kumar models for Hilbert modular surfaces have
Igusa--Clebsch invariants of the form
\[ (I_2, I_4, I_6, I_{10}) = \left(-24(B_1/A_1), -12A, 96 (A B_1/A_1) - 36B, -4A_1B_2 \right) \]
where $A, A_1, B, B_1, B_2$ are rational functions on
the corresponding Humbert surface.

Let $T^{(1)}$ be the Gram matrix for the IC-simplified Mestre conic 
$Q_1$ in terms of $A$, $A_1$, $B$, $B_1$, and $B_2$.  
Let $T^{(2)}$ be $\frac{A_1^3}2$ times the Gram matrix for $Q_1$ with
respect to the basis $\{ \frac 18 e_1, \frac{e_2}{36A_1} , \frac 1{24}e_3 \}$.
The upper triangular entries of $T^{(2)}$ are
\begin{align*}
T^{(2)}_{1,1} &= -225 A_{1}^{3} B + 285 A A_{1}^{2} B_{1} + 324 B_{1}^{3}
\\
T^{(2)}_{1,2} &= 20 A^{2} A_{1}^{2} - 45 A_{1} B B_{1} + 36 A B_{1}^{2}
\\
T^{(2)}_{1,3} &= 1170 A A_{1}^{3} B - 1050 A^{2} A_{1}^{2} B_{1} - 1125 A_{1}^{4} B_{2} + 486 A_{1} B B_{1}^{2} - 1296 A B_{1}^{3}
 \\
T^{(2)}_{2,2} &= -60 A A_{1} B + 4 A^{2} B_{1} + 125 A_{1}^{2} B_{2}
 \\
T^{(2)}_{2,3} &= -20 A^{3} A_{1}^{2} - 405 A_{1}^{2} B^{2} + 234 A A_{1} B B_{1} - 144 A^{2} B_{1}^{2} + 1350 A_{1}^{2} B_{1} B_{2}
 \\
T^{(2)}_{3,3} &= -4140 A^{2} A_{1}^{3} B + 3840 A^{3} A_{1}^{2} B_{1} + 4275 A A_{1}^{4} B_{2}\\ &\hspace{0.45cm}+ 729 A_{1}^{2} B^{2} B_{1} - 3888 A A_{1} B B_{1}^{2} + 5184 A^{2} B_{1}^{3}.
\end{align*}

Let $Q_2$ be the associated ternary quadratic form with
coefficients in $\Z[A,A_1, B, B_1, B_2]$.  Note that $A_1^4$ divides the discriminant of $Q_2$.
In attempting to kill off the $A^2 B_1^3$ term from $Q_2(e_3)$,
we can simplify $T^{(2)}_{3,3}$  quite a bit, and remove a factor of
$A_1^2$ from the discriminant.
Let $v_2 = 4A e_1 + e_3$.  Then 
\[ Q_2(v_2) = -27  A_{1}^{2}  (-60 A_{1} A^{2} B + 175 A_{1}^{2} A B_{2} - 27 B_{1} B^{2}), \]
so we let $T^{(3)}$ be the Gram matrix of $Q_2$ with respect to
$\{ e_1, e_2, \frac{v_2}{3A_1} \}$.   The upper triangular coefficients of $T^{(3)}$ are
\begingroup
\allowdisplaybreaks
\begin{align*}
T^{(3)}_{1,1} &= -225 A_{1}^{3} B + 285 A A_{1}^{2} B_{1} + 324 B_{1}^{3}
\\
T^{(3)}_{1,2} &= 20 A^{2} A_{1}^{2} - 45 A_{1} B B_{1} + 36 A B_{1}^{2}
\\
T^{(3)}_{1,3} &= 90 A A_{1}^{2} B + 30 A^{2} A_{1} B_{1} - 375 A_{1}^{3} B_{2} + 162 B B_{1}^{2}
 \\
T^{(3)}_{2,2} &= -60 A A_{1} B + 4 A^{2} B_{1} + 125 A_{1}^{2} B_{2}
 \\
T^{(3)}_{2,3} &= 20 A^{3} A_{1} - 135 A_{1} B^{2} + 18 A B B_{1} + 450 A_{1} B_{1} B_{2}
 \\
T^{(3)}_{3,3} &= 180 A^{2} A_{1} B - 525 A A_{1}^{2} B_{2} + 81 B^{2} B_{1}.
\end{align*}
\endgroup

We call the resulting quadratic form $Q_3$ the \emph{RM-simplified Mestre conic}.

Note that $A_1^2$ divides the discriminant of $Q_3$, and $A_1$ divides the diagonal minors of $T^{(3)}$.  While one can remove another factor of $A_1$ from the discriminant, we do not know how to do this without introducing other factors into the discriminant.

%%%%%%%%%%%%%%%%%%%%%%%%%%%%%%%%
%%%%%%%%%%%%%%%%%%%%%%%%%%%%%%%%
%
%  ALGORITHMS
%
%%%%%%%%%%%%%%%%%%%%%%%%%%%%%%%%
%%%%%%%%%%%%%%%%%%%%%%%%%%%%%%%%

\section{An algorithm for minim{\iz}ing a conic over \texorpdfstring{$\Q(t_1,t_2)$}{{\unichar{"211A}}(t{\textunderscore}1,t{\textunderscore}2)}}
\label{sec:algorithms}

Let $R$ be an integral domain and $k$ its field of fractions.
Suppose $L : Q(X,Y,Z) = 0$ is a conic over a polynomial ring 
$R[t_1, \dots, t_m]$.   
We say $L$ is \emph{minimal} if its discriminant has minimal degree among the
$k(t_1, \dots, t_m)$-equivalence classes of $L$.
We would also like a notion of a \emph{reduced minimal form}, to encapsulate the idea that the coefficients are also simple as possible.  If $L$ is minimal and diagonal, then the coefficient degrees add up to the discriminant degree, and this should be considered reduced.   

In this section we present algorithms to search for a reduced minimal
form of $L$ in a certain (algorithmically defined) subset of the 
$k(t_1, \dots, t_m)$-equivalence class of $L$.  These
algorithms are tailored to
our case of interest: Mestre conics for Hilbert modular surfaces
over $\Q$. For concreteness and ease of exposition, 
we will assume in what follows $L$ is a conic defined over $\Z[t_1, t_2]$.
However, our algorithms can be adapted to more general settings,
and much of it makes sense beyond the situation where $L$ is
a conic in $\mathbb P^2$ (see \Cref{rem:more-generality}).  

\begin{rem} \label{rem:redmin}
In general, it may not be easy to verify whether $L$ is minimal.
In our situation, where $L$ is the Mestre conic over a Hilbert modular surface $Y_-(D)$ with birational model $z^2 = \lambda_D$, \cref{prop:pD-is-a-norm} tells us that our Mestre conic can be put in the form $X^2 - D \lambda_D Y^2 - q_D Z^2$ for some polynomial $q_D \in \Z[g,h]$.  For our purposes, we consider such a Mestre conic to be in  \emph{reduced minimal form} if $q_D$ has no non-constant factors in $\Z[g,h]$ which are norms from $\Q(g,h)(\sqrt {D \lambda_D})$.  
\end{rem}

\subsection{Minim{\iz}ation}
\label{sec:conic-minim{\iz}ation-algo}
Let $A$ be a discrete valuation ring with maximal ideal $\mathfrak{p} = (\pi)$. Let $S/A$ be an affine scheme flat of relative dimension $1$ over $\Spec A$ and suppose that $S$ has smooth generic fibre. Explicitly we may take $S$ to be given by the vanishing of polynomials $f_1, ..., f_\ell \in A[x_1, ..., x_n]$ in $\mathbb{A}^n_{A}$ such that $f_i \not\in \pi A[x_1, ..., x_n]$ for all $i = 1,...,\ell$. Assume that the point $\mathfrak{m} = (\pi, x_1, ..., x_n)$ lies on $S$ (this is the origin on the special fibre of $S$). Following \cite[IV.7]{silverman-advanced} we define the {blow-up} of $S$ at $\mathfrak{m}$ to be the subscheme of $S \times \PP_{\Z}^{n}$ given by the equations
\begin{equation*}
  x_i y_0 - \pi y_i \quad \text{ and } \quad x_i y_j - x_j y_i  
\end{equation*}
where $1 \leq i,j \leq n$.

Just as curve and surface singularities (over $\C$) may be resolved by iterated blow-ups and normal{\iz}ations, we may hope that in our setting arithmetic blow-ups will (at least partially) resolve the singular points on the special fibre of $S$. Indeed, this approach is utilised in Tate's algorithm for finding the minimal regular model of an elliptic curve over a DVR~\cite[IV.7]{silverman-advanced}. We will refer to this process as \emph{minim{\iz}ation} in reference to the fact that when $E/\Z_{(p)}$ is an elliptic curve, repeatedly applying the above algorithm ``minim{\iz}es the exponent of $p$ in the discriminant of $E$''.

If $B$ is a unique factorisation domain and $\pi \in B$ is a prime element, we denote by $v_{\pi}(x)$ the $\pi$-adic valuation of an element $x \in B$.

\begin{ex}[An algorithm for minim{\iz}ing a conic in \texorpdfstring{$\PP^2_{\Q}$}{P2Q}]
  \label{ex:min-conic}
  In the following example we illustrate how arithmetic blow-ups allow us to reproduce the classical reduction algorithm for ternary quadratic forms over $\Q$. The approach we present here is a closely related to (a simplified form of) \cite[Algorithm~I]{CremonaRusin_conics}.

  Let $Q(X,Y,Z) \in \Z[X, Y, Z]$ be a (homogeneous) degree $2$ polynomial which defines a non-singular curve $C \subset \PP^2_{\Q}$. After possibly rescaling $Q$ we may assume that the Gram matrix has entries in $\Z$ and that $\Delta(Q) \in \Z$.
  
  For simplicity we avoid characteristic $2$ and consider $Q$ as defining a scheme $\mathcal{C} \subset \PP^2_{\Z[1/2]}$. For each prime number $p \neq 2$ the fibre of $\mathcal{C}$ at $p$ is singular if and only if $\Delta(Q) \equiv 0 \pmod{p}$. This suggests the following minim{\iz}ation algorithm:

  If $p^2 \mid \Delta(Q)$ we consider the scheme $\mathcal{C}_p = \mathcal{C} \times_{\Z[1/2]} \Spec \Z_{(p)}$. The singular subscheme of the special fibre at $p$ consists of either a point, or a double line. Let $\overbar{U} \in \SL_3(\F_p)$ be a matrix which transforms the singular locus to the point $(0:0:1) \in \PP^2_{\F_p}$, respectively the line $\{Z = 0\} \subset \PP^2_{\F_p}$. By the existence of Smith normal form, the matrix $\overbar{U}$ lifts to a matrix $U \in \SL_2(\Z)$, which we then apply to $\mathcal{C}$ by setting $M' = U^\mathrm{T} M U$. Since $\det(U) = 1$ this leaves $\Delta(Q)$ unchanged.

  In the former case, choosing the affine patch where $Z = 1$, and setting $x = X/Z$ and $y = Y/Z$ we blow-up the singular point $\mathfrak{m} = (p, x, y)$ on $\mathcal{C}_p$. Since $p^2 \mid \Delta(Q)$
  this is given (on an open subscheme of the blow-up) by the vanishing of the integral non-homogeneous quadratic form $Q(px, py,1)/p^2$. After homogenising we obtain the ternary quadratic form $Q'(X, Y, Z) = Q(X,Y,p^{-1}Z)$. In particular $\Delta(Q') = \Delta(Q)/p^2$.

  In the case of the double line at $Z = 0$, we set $Q'(X, Y, Z) = Q(X, Y, pZ)/p^k$ for some $k \in \{1,2\}$. We have $\Delta(Q') = p^{2-3k} \Delta(Q)$.
  
  It follows that after finitely many iterations $v_p(\Delta(Q)) \leq 1$. Repeating the above algorithm at each odd prime $p \mid \Delta(Q)$, then after finitely many steps we have $v_p(\Delta(Q)) \leq 1$. We refer to this as a \emph{minimal model for $Q$} (away from the prime $2$).
\end{ex}

It is important to note that the ``global'' approach in \Cref{ex:min-conic} may fail when the base ring (in that case $\Z$) fails to be a principal ideal
domain. More specifically, if $A$ is an integral domain and $\mathfrak{p}$ is a maximal ideal of $A$ it is not necessarily true (unless $A$ is a PID) that we may lift a matrix $\overbar{U} \in \SL_3(A/\mathfrak{p})$ to a matrix $U \in \SL_3(A)$.

Nevertheless, we will see that in practice applying \Cref{ex:min-conic} in the case when $A = \Z[t_1, t_2]$ provides a useful algorithm for minim{\iz}ing a ternary quadratic form $Q(X,Y,Z)$ with coefficients in $\Z[t_1,t_2]$, which we describe in \Cref{alg:desing}. Let $\pi \in \Z[t_1,t_2]$ be a prime factor of $\Delta(Q)$. While we cannot in general hope for a lift $U \in \SL_3(\Z[t_1,t_2])$ which moves the singular point (or line) on the special fibre to the origin (respectively $Z = 0$), we can always choose a lift $U \in \mathrm{M}_3(\Z[t_1,t_2]) \cap \GL_3(\Z[t_1,t_2]_{(\pi)})$. In this case we introduce a factor of $\det(U)^2$ into $\Delta(Q)$.

Let $L/\Z[t_1,t_2]$ be a conic given by a Gram matrix $M$ with coefficients in $\Z[t_1,t_2]$. Let $\pi \in \Z[t_1,t_2]$ be an irreducible element. We define $\overbar{L}_\pi$ to be the generic fibre of the reduction of $L$ modulo the ideal $(\pi)$. If $a,b,c \in \Z[t_1,t_2]$ we write $\diag(a,b,c)$ for the diagonal matrix with entries $a$, $b$, and $c$. For each matrix $U \in \mathrm{M}_3(\Z[t_1,t_2]) \cap \GL_3(\Q(t_1,t_2))$ we denote by $L^U/\Z[t_1,t_2]$ the conic with Gram matrix $U^{T} M U$.

\begin{algorithm}
  \caption{{\tt Minim{\iz}eAtPi}$(L,\pi)$: Minim{\iz}e the conic $L$ at a prime $\pi \mid \Delta(L)$. \\
    Input: A conic $L/\Z[t_1,t_2]$ and a prime element $\pi \in \Z[t_1,t_2]$ where $(\pi) \ne  (2)$ and $\pi^2 \mid \Delta(L)$. \\
    Output: A conic $L'/\Z[t_1,t_2]$ such that $v_{\pi}(\Delta(L')) \leq v_{\pi}(\Delta(L)) - 2$.
    }

  \label{alg:desing}
  \begin{algorithmic}
    \State $\mathtt{Sing} \gets$ the singular subscheme of $\overbar{L}_{\pi}$
    \If{$\dim \mathtt{Sing} = 0$}
      \State $\overbar{U} \gets$ a matrix in $\SL_3(\Z[t_1,t_2]/(\pi))$ which moves $\mathtt{Sing}$ to $(0:0:1)$
      \State $U \gets$ a lift of $\overbar{U}$ to $\mathrm{M}_3(\Z[t_1,t_2]) \cap \GL_3(\Z[t_1,t_2]_{(\pi)})$
      \State $L' \gets L^U$
      \State $V \gets \diag(1,1,\pi^{-1})$
      \State $L' \gets (L')^V$
    \Else
      \State $\overbar{U} \gets$ a matrix in $\SL_3(\Z[t_1,t_2]/(\pi))$ which moves $\mathtt{Sing}$ to $\{Z=0\}$
      \State $U \gets$ a lift of $\overbar{U}$ to $\mathrm{M}_3(\Z[t_1,t_2]) \cap \GL_3(\Z[t_1,t_2]_{(\pi)})$
      \State $L' \gets L^U$
      \State $V \gets \diag(1,1,\pi)$
      \State $k \gets v_{\pi}(L')$
      \State $L' \gets \pi^{-k} (L')^V$
    \EndIf
    \State \Return $L'$
  \end{algorithmic}
\end{algorithm}

\begin{rem}
  \label{rem:desing-order}
  Note that \Cref{alg:desing} will
typically increase $Y^2$ or $Z^2$-coefficient degrees if the
diagonal degrees (i.e., $X^2$, $Y^2$, and $Z^2$-coefficients) of $L$ are 
not in increasing order.  In practice we therefore assume the diagonal degrees
are increasing by permuting the variables $X, Y, Z$.
\end{rem}

\begin{rem}
  \label{rem:more-generality}
  \Cref{alg:desing} can be generalised to the case of a scheme $S/\Z[t_1,...,t_m]$ of relative dimension zero or one equipped with a ``bad prime element'' $\pi \in \Z[t_1,...,t_m]$. One such situation is when $S$ is relative dimension zero and given by a single homogeneous polynomial $f(X,Y) \in \Z[t_1,...,t_m][X,Y]$. This viewpoint is useful for simplifying birational models for (small degree) coverings $X \to \PP^m$, by considering $X$ to be a hypersurface in $\PP^1_{\Z[t_1,..,t_m]}$ cut out by $f(X,Y)$.

  Even more generally one may replace $\Z[t_1,...,t_m]$ with a geometrically integral base scheme $T$.
\end{rem}

\subsection{Searching for minimal models}

Let $L : Q(X,Y,Z) = 0$ be a conic over $\Q(t_1, t_2)$.  
By the diagonal coefficients of $L$, we mean the $X^2$, $Y^2$
and $Z^2$-coefficients of $Q$.

By rescaling, we may assume the coefficients of $Q$ lie
in $\Z[t_1, t_2]$ and have gcd 1.  If $\pi \in \Z[t_1, t_2]$ is a non-unit such that
$\pi^2$ divides a diagonal coefficient (e.g., the $X^2$-coefficient) 
and $\pi$ divides each coefficient involving the same variable
(e.g., the $XY$ and $XZ$-coefficients), then we replace that variable
with itself divided by $\pi$ (e.g., replace $X$ with $X/\pi$).
Furthermore, if $\pi$ divides two of the diagonal
coefficients and their cross-term coefficient (e.g., the $X^2$, $Y^2$ 
and $XY$-coefficients), then we scale 
the other variable (e.g., $Z$) by $\pi$ and divide the whole conic equation
by $\pi$.  If $L$ satisfies all these assumptions, we say $L$ is
\emph{minimal with respect to scaling transformations}, or for short, \emph{scale minimal}.
At each stage in our algorithm, we will assume our conics are scale 
minimal.

Our algorithm to search for a reduced minimal model for $L$ consists
of constructing a search tree of $\Q(t_1, t_2)$-equivalent conics.  At each
stage, three possible types of minim{\iz}ation operations are allowed:

\begin{enumerate}[label=(M\arabic*)]
\item \label{min1}
  minim{\iz}ation of the degree of $L$ (i.e., minim{\iz}ation ``at infinity''),

\item \label{min2}
  minim{\iz}ation at a rational factor from the discriminant, and

\item \label{min3}
  minim{\iz}ation at a polynomial factor from the discriminant.
\end{enumerate} 

There are two immediate issues. First, for general conics it may not be easy to determine when we have found a minimal form, but in our situation we employ the notion of reduced minimal as in \cref{rem:redmin}.
Second, this search tree may be infinite, since removing factors from the discriminant
can introduce other factors (as discussed in \Cref{sec:conic-minim{\iz}ation-algo}).  To address the
second issue, we place some restrictions on our search process which are based
on observations made after performing several minim{\iz}ations ``by hand.''

In particular, we observed:

\begin{enumerate}[label=(P\arabic*)]
\item \label{princ1}
Minim{\iz}ing conics tends to be easier when the rational part of the
discriminant is small.

\item  \label{princ2}
  Minim{\iz}ing conics tends to be easier when the sum of the
diagonal coefficient degrees is close to the discriminant degree.

\end{enumerate}

For a scale minimal conic $L$, let $\Delta(L) \in \Z[t_1, t_2]$ be
the discriminant.  We define $\Delta_\Q(L)$ to be the rational part of $\Delta(L)$, i.e.,
the content of $\Delta(L)$.  Write $\Delta(L) = \Delta_\Q(L) \prod \pi_i^{e_i}$,
where the $\pi_i \in \Z[t_1, t_2]$ are coprime irreducible polynomials.
Let $\Delta_1(L) = \prod_{e_i = 1} \pi_i$ be the ``power-free part'' of 
$\Delta(L)$, and $\Delta_2(L) = \prod_{e_i > 1} \pi_i^{e_i}$ be the power-full
part of $\Delta(L)$.  By the diagonal degree sum of $L$, denoted
$\diag \deg L$, we mean the
sum of degrees of the diagonal coefficients.
We define the \emph{degree score} of $L$ to be
\begin{equation*}
  \mathtt{DegScore}(L) := \deg \Delta_2(L) + \diag \deg L - \deg \Delta(L).
\end{equation*}
Then our second observation \ref{princ2} means that we want to work with
conics with low degree score. Note that a degree score of 0 corresponds to having squarefree discriminant and $ \diag \deg L = \deg \Delta(L)$.

\begin{algorithm}[t]
  \caption{{\tt Minim{\iz}ationSearch($L_0$)}: Search for a minimal model for $L_0$ by removing power-full factors from the discriminant. \\
    Input: A conic $L_0 / \Q(t_0, t_1)$. \\
    Output: A minimal model for $L_0$.
  }
  \label{alg:main}
\begin{algorithmic}
  \State {\tt visited} $\gets \{L_0\}$
  \State {\tt queue} $\gets \{L_0\}$
  \While{{\tt queue} is not empty}
    \If{there exists $L_f \in $ {\tt queue} with degree score 0}
      \State \Return $L_f$
    \EndIf
    \State $L \gets $ an element of {\tt queue} with minimal path score
    \State remove $L$ from {\tt queue}
    \State $L \gets $ {\tt RationalMinim{\iz}ation}$(L)$
    \State $L \gets $ {\tt DegreeMinim{\iz}ation}$(L)$
    \If{$L \not\in $ {\tt visited}}
      \State add $L$ to {\tt visited}
      \State add $L$ to {\tt queue}
    \Else
      \For{each irreducible polynomial $\pi \mid \Delta_2(L)$}
        \State $L' \gets $ {\tt PolynomialMinim{\iz}ation}$(L,\pi)$
        \If{$\mathrm{ord}_\pi \Delta(L') < \mathrm{ord}_\pi \Delta(L)$ and $L' \not\in $ {\tt visited}}
          \State add $L'$ to {\tt visited}
          \State add $L'$ to {\tt queue}
        \EndIf
      \EndFor
    \EndIf
  \EndWhile
\State \Return {\tt Fail}
\end{algorithmic}
\end{algorithm}

\subsubsection{The main algorithm}
First we outline our main algorithm, \texttt{Minim{\iz}ationSearch}, which we present in \Cref{alg:main}.
We describe pieces of the
algorithm in more detail later. The \texttt{Minim{\iz}ationSearch} algorithm creates a search tree where
the nodes are transformed conics, and will terminate if it finds a conic with
degree score 0.  The order in which the tree is searched depends on the
path score of each leaf $L$ in the tree. In computer science this type of search is known
as a best-first search. 
In \Cref{sec:scoring} we discuss several options for path scoring, but our default is the average slope score which is essentially 
the average rate of change of the degree score along the path from the
root node $L_0$ to the node $L$.

In general, the output of the algorithm may or may not be minimal, in the sense we have defined above.  However, in our situation, we diagonal{\iz}e the resulting conic to put it into the form given in \cref{prop:pD-is-a-norm}.  When necessary, one can remove norm factors from the last coefficient to reach a reduced minimal form.

\subsubsection{The sub-algorithms}
We now present the various sub-algorithms used in \cref{alg:main}.

Our first sub-algorithm, \cref{alg:ratred}, applies \Cref{alg:desing} to remove as many rational prime factors as possible from the discriminant of a conic $L$ without increasing its diagonal degrees.
The order in which rational factors are removed can make a difference, and in our implementation we minim{\iz}e starting with the largest prime $p$. In practice we observed that this performs better than starting with smaller primes.

Our second sub-algorithm, \cref{alg:degred}, decreases the diagonal degrees of $L$ by minim{\iz}ing at the place at infinity in both the affine patch where $t_1 = 1$ and the affine patch where $t_2 = 1$. For a conic $L$ with coefficients in $\Z[t_1,t_2]$ we define {\tt SwapAffinePatch}$(L, t_i)$ to be the function which homogen{\iz}es the coefficients of $L$ with a transcendental $t_3$, swaps $t_i$ and $t_3$, dehomogen{\iz}es the resulting coefficients over $\Z[t_1,t_2]$, and returns the scale minimal form of the resulting conic.

Our final sub-algorithm, \cref{alg:polyred}, removes an irreducible
polynomial factor $\pi$ from the discriminant of $L$ without increasing the
degree score.
\begin{algorithm}[p]
  \caption{{\tt RationalMinim{\iz}ation(L)}: Minim{\iz}e $L$ at rational primes $p \mid \Delta_\Q(L)$.\\
    Input: A conic $L/\Q(t_1, t_2)$.\\
    Output: A model $L'$ for $L$, obtained by minim{\iz}ing at rational primes subject to the condition that $\diag \deg L' \leq \diag \deg L$.
  }
  \label{alg:ratred}
\begin{algorithmic}
\State $D \gets \Delta_\Q(L)$
\For{$p \mid D$}
  \While{$p^2 \mid \Delta_\Q(L)$}
    \State $L' \gets$ {\tt Minim{\iz}eAtPi}$(L,p)$
    \If{$|\Delta_\Q(L')| < |\Delta_\Q(L)|$ and $\diag \deg L' \le \diag \deg L$}
      \State $L \gets L'$
    \Else
      \State {\tt break}
    \EndIf
  \EndWhile
\EndFor 
\State \Return $L$
\end{algorithmic}
\end{algorithm}

\begin{algorithm}[p]
  \caption{{\tt DegreeMinim{\iz}ation(L)}: Minim{\iz}e $L$ at the place at infinity to decrease its degree.\\
    Input: A conic $L/\Q(t_1, t_2)$.\\
    Output: A model $L'$ for $L$, obtained by minim{\iz}ing at the place at infinity subject to the condition that $\diag \deg L' \leq \diag \deg L$.}
  \label{alg:degred}
  \begin{algorithmic}

    \For{$i \in \{1, 2 \}$}
      \State $L_i \gets $ {\tt SwapAffinePatch}$(L, t_i)$
      \While{$t_i \mid \Delta_2(L_i)$}
        \State $L'_i \gets $ {\tt Minim{\iz}eAtPi}$(L_i,t_i)$
        \If{$\diag \deg L'_i \le \diag \deg L_i$}
          \State $L_i \gets L_i'$
        \Else
          \State {\tt break}
        \EndIf
      \EndWhile
      \State $L_i \gets $ {\tt SwapAffinePatch}$(L_i', t_i)$
    \EndFor
    \State \Return the first element of $( L, L_1, L_2 )$ which minim{\iz}es $\diag \deg$
  \end{algorithmic}
\end{algorithm}

\begin{algorithm}[p]
  \caption{{\tt PolynomialMinim{\iz}ation($L, \pi$)}: minim{\iz}e $L$ at a prime $(\pi) \subset \Q[t_1,t_2]$.\\
    Input: A conic $L/\Q(t_1, t_2)$.\\
    Output: A model $L'$ for $L$, obtained by minim{\iz}ing at the prime $(\pi)$ subject to the condition that $\mathtt{DegScore}(L') \leq \mathtt{DegScore}(L)$.
  }
  \label{alg:polyred}
  \begin{algorithmic}

    \State $L' \gets $ {\tt Minim{\iz}eAtPi}$(L,\pi)$
    \If{$\mathtt{DegScore}(L') \le \mathtt{DegScore}(L)$}
      \State $L \gets L'$
    \EndIf
    \State \Return $L$

  \end{algorithmic}
\end{algorithm}

\begin{rem} \label{rem:main-algo}
  We comment on our implementation of Algorithms~\ref{alg:main}--\ref{alg:polyred}.
  \begin{enumerate}
  \item
    When trying Algorithms~\ref{alg:ratred}--\ref{alg:polyred},
    we will try these algorithms on certain forms of $L$. 
    By \cref{rem:desing-order}, we want the diagonal degrees of $L$ to be increasing, and so we try
    these algorithms on every permutation of $\{ X, Y, Z \}$ such that the
    diagonal degrees are increasing.  The resulting conic can be significantly
    more complicated depending on the permutation used.  To keep the
    number of branches small, we only keep the resulting conic from one of
    these permutations, and it will be one with a minimal degree score.

  \item
    Both \Cref{alg:main} and the sub-algorithms~\ref{alg:ratred}--\ref{alg:polyred} make certain
    choices about the order of our three minim{\iz}ation operations~\ref{min1}--\ref{min3}, and when
    to no longer pursue certain search paths.  In practice, we avoid
    sequences of operations which make the conic worse along the way.
    One can modify these algorithms to include more branches and
    and be less greedy by not fixing the order of minim{\iz}ation operations and 
    by allowing operations which make the conic worse.  While this
    may allow us to find solutions we would not otherwise, in moderately complicated situations that we
    tested this less restrictive search tended
    to take much longer to complete.

  \item
    \label{enum:alg-random}
    One can also random{\iz}e \Cref{alg:main} to help mitigate getting stuck
    in unproductive sections of the search tree.  Namely, with some fixed probability, we choose $L$ uniformly at random among the leaves in the queue, as opposed to choosing one with minimal path score.  This random{\iz}ation sometimes speeds up the search process.

  \item
    \label{enum:rat-min-deg-min-option}
    Sometimes a good choice of polynomial minim{\iz}ation is not immediately
    apparent in the node score.  To help identify such branches, we have
    also implemented a variant of the algorithm where after each polynomial
    minim{\iz}ation, we immediately run {\tt RationalMinim{\iz}ation} and 
    {\tt DegreeMinim{\iz}ation}.  Sometimes this is slightly slower, and sometimes
    it is significantly faster.

  \end{enumerate}
\end{rem}

\subsection{Scoring methods}
\label{sec:scoring}
\cref{alg:main} relies on a path score for each node to choose the next
leaf in the search process.  If we merely used the degree score, our
search would be very slow (and potentially not terminate) when
there is no good search path along a branch that starts with 
minimal degree scores.  The path score (as well as random{\iz}ation) 
provides a balance between a purely greedy search and a
breadth-first search.

First we define the \emph{node score} of a node $L$ to be its degree score
plus the number primes $p \mid \Delta_\Q(L)$ dividing the rational part
of the discriminant.  This modification of the degree score is to account for
principle \ref{princ1}.  With this node score in mind, we consider
the following methods to define a path score.  A lower path score is considered
better.

\begin{itemize}

\item
  \emph{Average slope score.} The path score of $L$ is the difference
  between the node scores of $L$ and the root $L_0$, divided by the number
  of nodes on the path from $L_0$ to $L$.  

\item
  \emph{Penal{\iz}ed node score.} The path score of $L$ is the
  node score plus a penalty which depends on the length of the path.  Let $n$ be the number of nodes on the
  path from the root $L_0$ to $L$, excluding $L$ itself, whose 
  node score is the same as the node score of $L$.  We set the penalty
  to be $\frac{n^2}4$. 

\item
  \emph{Alternating score.}  Alternate the path score between the
  average slope and penal{\iz}ed node score methods.
\end{itemize}

The average slope score measures the rate at
which the node score is decreasing, and prevents the search from
spending too much time along paths where the node score does not improve
much or at all.  The penal{\iz}ed node score is closer to the greedy
approach of only using the degree score (or rather the node score), 
but which, at least temporarily, avoids paths along which the path score does not improve 
at all after a few steps.  The alternating score blends these two approaches.

%%%%%%%%%%%%%%%%%%%%%%%%%%%%%%%%
%%%%%%%%%%%%%%%%%%%%%%%%%%%%%%%%
%
%  RESULTS
%
%%%%%%%%%%%%%%%%%%%%%%%%%%%%%%%%
%%%%%%%%%%%%%%%%%%%%%%%%%%%%%%%%

\section{The output of \texorpdfstring{\Cref{alg:main}}{Algorithm 5.5}}
\label{sec:results}
For each positive fundamental discriminant $D < 100$, Elkies--Kumar~\cite{Elkies-Kumar} give a rational parametrisation of the Humbert surface $\mathcal{H}_D$, together with a rational function $\lambda_D \in \Q(g,h) \cong \Q(\mathcal{H}_D)$ such that the Hilbert modular surface $Y_{-}(D)$ is birational to the affine surface cut out by the vanishing of $z^2 - \lambda_D$ in $\mathbb{A}^3$. For each such discriminant, we apply \Cref{alg:main} to try to transform the (IC or RM simplified; see \cref{sec:simplifications}) Mestre conic $L_D / \Q(g,h)$ into the form $X^2 - D \lambda_D Y^2 - q_D Z^2$, for some rational function $q_D \in \Q(g,h)$ (such a model is guaranteed to exist by \Cref{prop:pD-is-a-norm}).

For $D < 100$, we first pre-compute a list of ``nice'' changes of
coordinates which minim{\iz}es and reduces a factor of $\Delta(L_D)$, or
the quantities $(A_1, A, B_1, B, B_2)$ of Elkies--Kumar (see \Cref{rem:nice-change-coord}). For each such
change of variables, we run our algorithm for each of our 3 scoring methods
for up to 48 hours.
 We also run the random{\iz}ed version explained in \cref{rem:main-algo}\eqref{enum:alg-random},
taking the random{\iz}ation probability $p = \frac 18$.  

Runtimes of successful cases for the first two
scoring methods are summar{\iz}ed in \cref{tab:runtimes}.
The second and third columns in \cref{tab:runtimes} give the discriminant degrees and
coefficient degrees of the initial Mestre conic $L_D$, as a measure of
complexity.  The next 2 columns list runtimes for the deterministic version
of our algorithms with the average slope score and 
penal{\iz}e node score methods.  The last 2 columns
list average runtimes
(over 5 trials) for the random{\iz}ed version of these 2 scoring methods. When $D=33,53,61$ a change of variables was used and we record
the runtimes using the change of variables that finished fastest for that
scoring method.

Note the runtimes are wall times, not CPU times, so differences of a few
seconds should be considered random noise.
Calculations were run on the OSCAR supercomputer at Brown University.

In \cref{tab:runtimes} bolded runtimes note where one (non-random{\iz}ed) scoring method significantly
outperformed the other.
Note that sometimes the average slope score is much better than the
penal{\iz}ed node score, and sometimes the converse is true.  
An extreme example is $D=44$ which completes in under 12 minutes for
the average slope score but does not finish before the 48-hour timeout
for the penal{\iz}ed node score.  The issue in this case is the penal{\iz}ed slope
method gets stuck on a single minim{\iz}ation computation.  In spite of this, with a suitable initial change of variables it finishes in just over 31 minutes.  Random{\iz}ed versions also terminate, and the variant 
of the algorithm in \cref{rem:main-algo}\eqref{enum:rat-min-deg-min-option} finishes in 11 minutes.
We also tested the alternating scoring method, and found it is faster for $D=61$
(4h\,7m), and while it often performs in between the other two scoring methods,
it is much worse when $D = 24, 33, 37$.

\begin{rem}[``Nice'' changes of coordinates]
  \label{rem:nice-change-coord}
  In some cases (namely when $D=33$, $53$, and $61$) we first needed to apply a projective linear change of coordinates to the model for $\mathcal{H}_D$ for our algorithm to successfully terminate. We arrive at several changes of coordinates using the \texttt{Magma} function \texttt{MinRedTernaryForm} developed by Elsenhans--Stoll~\cite{ElsenhansStoll_min}, and our own more naive reduction algorithms (which simply place the most singular points of a plane curve at infinity).
\end{rem}

Often it happens that, for a given $D$, one random instance
may result in a significant speed-up, but another random instance runs much
slower, so on average no time is saved.  The cases where random{\iz}ation 
usually results in a speed up are often the cases where many steps are required
in the deterministic version (see \cref{tab:steps}).  An extreme case is  $D=33$ using 
the penal{\iz}ed node score. In this case the deterministic version runs in 420 steps and
takes about 5 times a long as a typical random instance.
For $D = 44, 53, 61$, of 5 random{\iz}ed trials using the penal{\iz}ed node score, 3, 2 and 1 instances,
respectively, did not finish before the 48-hour timeout.

\begin{table}
  \begin{tabular}{l||c|c||c|c||c|c}
    $D$ & $\deg \Delta_{L_D}$ & $\deg L_D$ & slope          & pen.\ node     & rand. slope & rand. node  \\
    \hline
    5   & 11                  & 7          & 1s             &  1s            & 1s          & 1s          \\
    8   & 13                  & 7          & 6s             & 6s             & 6s          & 6s          \\
    12  & 29                  & 15         & 1m\,21s        & {\bf 1m\,1s}   & 1m\,29s     & 1m\,28s     \\
    13  & 20                  & 10         & 1m\,25s        & 1m\,29s        & 1m\,29s     & 1m\,30s     \\
    17  & 32                  & 16         & 48s            & {\bf 29s}      &  1m\,1s     & 30s         \\
    21  & 38                  & 16         & 2m\,24s        & {\bf 2m\,3s}   & { 2m\,9s}   & 2m\,12s     \\
    24  & 43                  & 18         & {\bf 19s}      & 28s            &  24s        &  43s        \\
    28  & 45                  & 17         &  2h\,39m       & {\bf 15m\,30s} & { 2h\,18m}  & 18m\,27s    \\
    29  & 40                  & 18         & {\bf 1m\,52s}  & 3m\,1s         & 2m\,10s     &  2m\,50s    \\
    33* & 50                  & 26         & {\bf 23m\,6s}  & 1h\,22m        & 30m\,46s    & { 15m\,55s} \\
    37  & 48                  & 18         & {\bf 12m\,21s} & 49m\,20s       & 12m\,48s    & { 40m\,16s} \\
    44  & 96                  & 42         & {\bf 11m\,42s} & $>$2d          & 13m\,4s     & $>$28h      \\
    53* & 70                  & 30         & {\bf 13m\,32s} & 2h\,50m        & 58m\,50s    & $>$20h      \\
    61* & 78                  & 32         & {\bf 4h\,41m}  & 7h\,49m        & { 4h\,10m}  &  $>$13h
  \end{tabular}
\caption{Approximate runtimes for minim{\iz}ing $L_D$ with different scoring functions. Starred discriminants required initial change of variables and bold values correspond to cases where one scoring method significantly outperformed another.}
\label{tab:runtimes}
\end{table}

\begin{table}
\begin{tabular}{l|c|c||c|c||c|c}
& \multicolumn{2}{c||}{minim{\iz}ation primes} & \multicolumn{2}{c||}{slope score}
& \multicolumn{2}{c}{pen.\ node score} \\ 
\hline
$D$  & $\# \{ \frakp^2 \mid \Delta_L \}$ & $\sum \lfloor \frac{v_\frakp(\Delta_L)}2 \rfloor$ & steps &  depth & steps & depth \\
\hline
5 & 2 & 2 & 2 & 2 & 2 & 2 \\
8 & 3 & 3 & 4 & 4 & 4 & 4  \\
12 & 5  & 5 & 14 & 8 & {\bf 9} & 8 \\
13 & 4 & 4 & 8 & 7 & {\bf 7} & 7 \\
17 & 5 & 6 & 22 & 10 & {\bf 11} & 10 \\
21 & 6 & 8 & 19 & 12 & {\bf 13} & {\bf 10} \\
24 & 8 & 9 & {\bf 14} & 12 & 51 & 12 \\
28 & 8  & 11 & 1053 &  21 & {\bf 56} & 21 \\
29 & 7 & 9 & {\bf 16} & {\bf 9} & 24 &  15  \\
33* & 8 & 11 & {\bf 72} & {\bf 18} & 420  & 23 \\
37 & 8 & 10 & {\bf 90} & 17 & 241 & 17 \\
44 & 10 & 12 & {\bf 74} & {\bf 24} &  --- & --- \\
53* & 10 & 13 & 56 & {\bf 19} & 56 & 20 \\
61* & 11 & 15 & 413 & {\bf 21} & {\bf 396} & 22 \\
\end{tabular}
\caption{Numbers of steps to minim{\iz}e $L_D$ and depth of solutions. Starred discriminants required an initial change of variables and bold values correspond to cases where one scoring method outperforms another.}
\label{tab:steps}
\end{table}

To get a better sense of how difficult it was to find a sequence of transformations
to minim{\iz}e $L_D$, compare with \cref{tab:steps}.  The second
column lists the number of polynomial factors of $\Delta_L$ occurring to at
least a square power, i.e., the number of prime ideals $\frakp$ in $\Q[g,h]$ where 
one needs to minim{\iz}e.  Sometimes these factors occur to higher powers, and
the third column counts these with appropriate multiplicity.  In particular, the third
column in \cref{tab:steps} tells us the smallest number of minim{\iz}ation steps we expect to need to perform to carry out the minim{\iz}ation completely. 
The fourth column reports, when using the average slope score, 
the number of steps (i.e., number of times the {\tt while} loop is iterated) 
required to complete {\tt Minim{\iz}ationSearch}.  The fifth column, again
for average slope score, reports the depth of the final solution in the search tree. 
The last two columns report analogous data when using the penal{\iz}ed
node score.  

We note that runtimes and number of steps required are not perfectly
correlated, as certain minim{\iz}ation steps take much longer to run than others
(e.g., compare the average slope with penal{\iz}ed node scores for $D=33$ or $D=61$).
We expect that the best search paths typically avoid the most 
intensive minim{\iz}ation calculations.  Thus, using automated
timeouts in the search process would likely increase efficiency (we have not implemented this).

It is also interesting to note that for most $D \ge 29$, the two 
scoring methods are finding different paths to the minimal model (the depths
are typically different).

\section{Models for genus 2 curves with RM}
\label{sec:proof-main-results}

Using the transformations computed as described in \Cref{sec:results} we now prove our main results about generic models for genus $2$ curves with RM $D$.

\subsection{Proofs of \texorpdfstring{Theorems~\ref{thm:models} and \ref{thm:pDs}}{Theorems 1.1 and 1.5}}

\begin{proof}[Proof of \Cref{thm:pDs}]
  The polynomials $q_D \in \Q[g,h]$ are computed by applying the transformations calculated using \Cref{alg:main} to the generic Mestre conic $L_D$ which is defined over $\Q(\mathcal{H}_D) \cong \Q(g,h)$. These transformations are too complicated to reproduce here, but are stored in the electronic data accompanying this article at~\cite{electronic}.

  To determine the analogous polynomials $p_D$, we convert to Elkies--Kumar $(m,n)$-coordinates and again apply \Cref{alg:main}.  These transformations are also recorded  at~\cite{electronic}.
\end{proof}

Let $M_D / \Q(Y_{-}(D))$ be the Mestre cubic defined in \Cref{sec:mestre-conic} associated to the generic point on the Hilbert modular surface $Y_{-}(D)$. Recall that we write $\widetilde{L}_D / \Q(Y_{-}(D))$ for the transformed Mestre conic given by $X^2 - DY^2 - q_DZ^2 = 0$.

To deduce \Cref{thm:models} it remains to apply Mestre's result (\Cref{prop:conic-cubic}) to $\widetilde{L}_D$ and the corresponding transformed cubic. Note that $\widetilde{L}_D$ may not have a point (except when $D \equiv 1 \pmod{8}$, see \Cref{thm:qD=1}). To overcome this, we note that a point on the threefold $\mathscr{L}_D$ defined in \Cref{thm:models} allows us to parametr{\iz}e the conic $\widetilde{L}_D$, and hence recover the Weierstrass models in \Cref{thm:models} via \Cref{prop:conic-cubic}.

\begin{proof}[Proof of \Cref{thm:models}]
  Let $\widetilde{R}_D(X,Y,Z) \in \Q(Y_{-}(D))[X,Y,Z]$ be the homogeneous cubic form which defines the cubic curve obtained by applying the transformations stored in \cite{electronic} to the Mestre cubic $M_D$.

  Consider a $\Q(\mathscr{L}_D)$-rational parametr{\iz}ation $\mathbb{A}^1 \dashrightarrow \widetilde{L}_D$ given by $x \mapsto [\eta_0 : \eta_1 : \eta_2]$ for some polynomials $\eta_i \in \Q(\mathscr{L}_D)[x]$. Our choice of parametrisation was computed by stereographic projection away from the point $(r, s, 1) \in \widetilde{L}_D(\Q(\mathscr{L}_D))$ and is recorded in \cite{electronic}.

  It is simple to check using computer algebra (e.g., \texttt{Magma}) that $\widetilde{R}_D(\eta_0, \eta_1, \eta_2)$ is equal to $F_D(z,g,h,r,s; x)$ up to a constant factor in $\Q(\mathscr{L}_D)$. The claim in (i) follows immediately from \Cref{prop:conic-cubic}, together with \cite[Proposition~2.1]{CM:RM5}.

  For (ii), note that $(g,h)$ are coordinates for the Elkies--Kumar model for the Humbert surface $\mathcal{H}_D$ (which is birational to a subvariety of $\mathcal{M}_2$).
\end{proof}

\subsection{Models when \texorpdfstring{$Y_-(D)$}{Y{\textunderscore}(D)} is rational} \label{sec:3-param-models}
As described in the introduction, when $D = 5$, $8$, $12$, $13$, and $17$ the threefold $\mathscr{L}_D$ is rational, which we now prove. In particular, we give generic families of genus $2$ curves with RM $D$ in three parameters with no relations. When $D=5$ this model is given in \cite[Remark 6.2]{CM:RM5}.

\begin{cor}
  \label{thm:RMrat}
  Let $k$ be a field of characteristic $0$. For each $D \in \{5,8,12,13,17\}$ and $a,b,c \in k$ consider the degree $6$ polynomial $f_{D}(a,b,c;x) \in k[x]$ recorded in \cite{electronic}. If $f_{D}(a,b,c;x)$ has no repeated roots, then the Jacobian $J$ of the genus $2$ curve $C/k$ with Weierstrass equation $C : y^2 = f_{D}(a,b,c;x)$ has RM $D$ over $\bar k$. Moreover if $\End_{\overbar k}(J)$ is abelian, then the RM is defined over $k$.
  \end{cor}

\cref{thm:RM12} and \cref{thm:RM17}  explicate the $D=12$ and $D=17$
cases of \cref{thm:RMrat}.  
%%%%%%%%%%%%%%%%%%%%%%%%%%%%%%%% 
% DISCRIMINANT 8
%%%%%%%%%%%%%%%%%%%%%%%%%%%%%%%% 
When $D=8$ we have $f_{8}(a,b,c;x) = \Nm_{L/K} \poly(x)$, where $L = k[r]/\xi(r)$, 
\begin{equation*}
  \xi(r)   =  (-a^2 + 2b^2 - 1) r^3 - 3c r^2 + (4a^4 - 16a^2b^2 + 2a^2 + 16b^4 - 4b^2 - 2c^2 - 2) r - 2c,
\end{equation*}
and
\begin{flalign*}
  \poly(x) = &\,\, (2 ( 2 b - 1 ) ( a^{2} - 2 b^{2} + 1 )r^2 + 4 c ( a^{2} - 2 b^{2} + 2 b - 1 )r - 4 ( 4 a^{4} b - 2 a^{4} + 2 a^{3} c            & \\
             &- 16 a^{2} b^{3} + 8 a^{2} b^{2} + 2 a^{2} b - a^{2} - 4 a b^{2} c + 16 b^{5} - 8 b^{4} - 4 b^{3} + 2 b^{2} - 2 b + 1 ) x^2          & \\
             &+4 ( a ( a^{2} - 2 b^{2} + 1 )r^2 + 2 ac r -  2 ( 2 a^{5} - 8 a^{3} b^{2} + a^{3} + 2 a^{2} b c + 8 a b^{4} - 2 a b^{2} - a          & \\
             &- 4 b^{3} c )) x + ( 2 b + 1 ) ( a^{2} - 2 b^{2} + 1 )r^2 -  2 c ( a^{2} - 2 b^{2} - 2 b - 1 )r -  2 ( 4 a^{4} b + 2 a^{4}           & \\
             &+ 2 a^{3} c - 16 a^{2} b^{3} - 8 a^{2} b^{2} + 2 a^{2} b + a^{2} - 4 a b^{2} c + 16 b^{5} + 8 b^{4} - 4 b^{3} - 2 b^{2} - 2 b - 1 ). &
\end{flalign*}

\begin{rem}
  The concise presentations given for $D=8$ and $D=12$ (with the Weierstrass sextic being given as a norm from a cubic {\'e}tale $k$-algebra) is a general phenomenon for genus $2$ curves whose Jacobians admit a Richelot isogeny (see e.g., \cite[Lemma~4.1]{Bruin-Doerksen}). In particular, every genus $2$ curve with RM by an order of discriminant $D \equiv 0 \pmod{2}$ admits a model of this form.

  Similarly the simple presentation in \Cref{thm:RM17} follows from \Cref{lemma:1-mod-8-factor}. In contrast, we do not know of any simpler presentation when $D \equiv 5 \pmod{8}$. It would be interesting to simplify the models $C : y^2 = f_D(a,b,c;x)$ which we give when $D = 5, 13$ (possibly by developing a satisfactory algorithm for minim{\iz}ing a coupled conic-cubic pair in $\mathbb{P}^2$).
\end{rem}

\begin{proof}[Proof of \cref{thm:RMrat}]
  This follows from \Cref{thm:pDs}, analogously to \Cref{thm:models}. 
  Let $\widetilde{L}_D' : X^2 - DY^2 - p_DZ^2 = 0$, where $p_D$ is as in \cref{tab:pD},
  and let $\widetilde{R}_D'(X,Y,Z) \in \Q(m,n)[X,Y,Z]$ be the associated cubic form (see \cite{electronic}).
  Let $\mathscr{L}_D' : r^2 - Ds^2 - p_D = 0$ be the Mestre conic bundle over $\mathbb A^2_{m,n}$.

  In the electronic data associated to this article \cite{electronic} we record rational parametr{\iz}ations $(a,b,c) \mapsto (m,n,r,s)$ of the threefolds $\mathscr{L}_D'$ for each $D \leq 17$.

  Let $x \mapsto [\nu_0: \nu_1 : \nu_2]$ be the $\Q(a,b,c)$-rational parametrisation of  $\widetilde{L}_D'$ recorded in \cite{electronic}. When $D = 17$ such a parametrisation is given over $\Q(a,b)$. It is simple to check with computer algebra that $\widetilde{R}_D'(\nu_0, \nu_1, \nu_2)$ is equal to $f_D(a,b,c;x)$ up to a constant factor, and the claims follow from \Cref{prop:conic-cubic}.
\end{proof}  
Indeed, the proof of \cref{thm:RMrat} also allows us to recall the forgetful maps from the parameter space to the Humbert surface $\mathcal{H}_D$, and we immediately deduce the following.

\begin{prop}
  \label{prop:when-isom}
  For $D \in \{5, 8, 12, 13, 17 \}$, let $f_D(a,b,c; x) \in \Z[a,b,c][x]$ be the polynomials defined in \cref{thm:RMrat} (and it is understood that $c$ plays no role when $D = 17$). The natural map from the parameter space $\mathbb{A}^3$ to the Humbert surface $\mathcal{H}_D$ which associates to a point $(a,b,c) \in \mathbb{A}^3$ the genus $2$ curve $C : y^2 = f_D(a,b,c;x)$ is given by
  \begingroup
  \allowdisplaybreaks
  \begin{flalign*}
    \mathrlap{\,\,(g,h)}\hphantom{(m,n)} &=
    \begin{dcases}
      \left(
      \frac{m^2 - 5n^2 - 9}{30},
      \frac{25(m+5)n^4 -5 \beta n^2 + (m + 3)^3(m - 2)^2}{12500}
      \right)
      & \text{if $D = 5$,} \\
      \mathrlap{
          \left(
          \frac{  ( m - 1 ) ( m + 1 )}{ 16 ( 2 n^{2} - 1 )},
          \frac{-32n^4 + 8(2m^2 + 7m + 9)n^2 - (m+3)^3}{ 16 ( 2 n^{2} - 1 ) ( m + 1 )}
          \right)
      }
      \hphantom{
        \left(
        \frac{8 (13 b^2 - a^2 + c^2 - 6 c - 3)}{a^2 - 13b^2 - c^2 + 3c - 3}, 
        \frac{24 (5a^2 - 65b^2 - 31c^2 - 34c - 9)}{a^2 - 13 b^2 - c^2 + 3 c - 3}
        \right)
      }
      & \text{if $D = 8$,}  \\[2mm]
      \left(
      \frac{2(m-1)(m+1)(m^2n+9m^2-8)}{27m^2- n^2 - 27},
      m
      \right)
      & \text{if $D = 12$,} \\[2mm]
      \begin{aligned}
        \bigg(
        &\frac{2(m^3 + 150m^2 - 6(n - 44)m - 16(9n + 4))}{9 (n^2 - 12m^3 + 3m^2)},\\
        &\frac{267m^3 - 24 (3 n - 148)m^2 + (n^2 - 1440n - 768)m + 128 n^2}{54(n^2 - 12 m^3 + 3 m^2)}
        \bigg)        
      \end{aligned}
      & \text{if $D = 13$,} \\[2mm]
      \left(
      \frac{\gamma}{3},
      \frac{(5 \gamma + 9) a + 6(\gamma + 1)}{18(2a - 3)}
      \right)
      & \text{if $D = 17$} \\[2mm]
    \end{dcases}&
  \end{flalign*}
  where
  \begin{equation*}
    \beta = 2m^3 + 10m^2 - 5m - 45,
  \end{equation*}
  \begin{equation*}
    \gamma = \frac{ - 3 ( 4 a^{3} - 6 a^{2} b + 12 a^{2} + 2 a b^{2} + 5 a b + 7 a - 3 b^{2} + 6 b + 1 )}{ 2 ( 4 a^{3} - 4 a^{2} b + 12 a^{2} + 2 a b^{2} + 9 a - 3 b^{2} + 9b )},
  \end{equation*}
  and

  \begin{flalign*}
      (m,n) &=
      \begin{dcases}
        \left(\frac{2(5a^2 + 5ac + b^2 - 5c^2 + 1)}{5a^2 - b^2 + 5c^2 - 1} , \frac{-b(4a + 2c)}{5a^2 - b^2 + 5c^2 - 1} \right)
        & \text{if $D = 5$,} \\
        \left( \frac{-2a^2 + 4b^2 + c^2 - 2}{2(a^2 - 2b^2)}, \frac{-c}{2} \right)
        & \text{if $D = 8$,} \\
        \left( \frac{-a^2 + 3b^2 + 3c^2 + 1}{a^2 - 3b^2 + 3c^2 - 1}, \frac{18c}{a^2 - 3b^2 + 3c^2 - 1} \right)
        & \text{if $D = 12$,} \\
        \left(
        \frac{8 (13 b^2 - a^2 + c^2 - 6 c - 3)}{a^2 - 13b^2 - c^2 + 3c - 3}, 
        \frac{24 (5a^2 - 65b^2 - 31c^2 - 34c - 9)}{a^2 - 13 b^2 - c^2 + 3 c - 3}
        \right)
        & \text{if $D = 13$.} \\
      \end{dcases}&
  \end{flalign*}
  \endgroup
  In particular, for a generic pair $P, P' \in \mathbb{A}^3(k)$ the associated genus $2$ curves are isomorphic over $\overbar{k}$ if and only if the image of $P$ and $P'$ under the above map are equal.
\end{prop}

\section{Relation to known families}
\label{sec:relat-prev-work}

\subsection{Discriminant \texorpdfstring{$8$}{8}}
\label{sec:prev-work-discr-8}
Bending \cite{Bending_thesis,Bending_paper} gave a versal family of genus $2$ curves $C/k$ with RM 8 by analysing the Richelot isogeny $\sqrt{2} \colon \Jac(C) \to \Jac(C)$. This family is given in terms of three parameters $A,P,Q$. This construction therefore induces a dominant rational map $\mathbb{A}^3 \dashrightarrow Y_{-}(8)$. By interpolating triples of $A,P,Q$ it is simple to recover this rational map in terms of Elkies--Kumar's~\cite{Elkies-Kumar} model for $Y_{-}(8)$. The rational map $\mathbb{A}^3 \dashrightarrow Y_{-}(8)$ is given by taking $(A,P,Q) \mapsto (m,n)$.
\begin{equation*}
  m = \frac{f_1(A,P,Q)}{f_3(A,P,Q) f_4(A,P,Q)^2} \qquad \text{and} \qquad n = \frac{-f_2(A,P,Q)}{f_3(A,P,Q) f_4(A,P,Q)}.
\end{equation*}
Here, the polynomials $f_i(A,P,Q) \in \Z[A,P,Q]$ have degrees $20$, $15$, $10$, and $5$ for each $i = 1, 2, 3, 4$ respectively. The precise formulae are too complicated to reproduce here, but we include them in the electronic data associated to this article~\cite{electronic}.

\subsection{Discriminant \texorpdfstring{$12$}{12}}
\label{sec:prev-work-discr-12}
Denote by $Y_{-}(12)[\sqrt{3}]$ the Hilbert modular surface of discriminant $12$ with \emph{full $\sqrt{3}$-level structure}. That is, $Y_{-}(12)[\sqrt{3}]$ is the surface whose $k$-points parametrise genus 2 Jacobians $J/k$ with an RM 12 action $\iota \colon \mathcal O_{12} \to \End_k^\dagger(J)$, such that $\ker \iota(\sqrt{3})$ is contained in $J(k)$.

Bruin, Flynn, and Shnidman~\cite{Bruin-Flynn-Shnidman} computed an explicit rational parametrisation $\PP^2 \dashrightarrow Y_{-}(12)[\sqrt{3}]$ over $\Q$ and gave formulae for the generic genus $2$ curve. By forgetting the level structure we obtain a natural forgetful morphism $Y_{-}(12)[\sqrt{3}] \to Y_{-}(12)$. This forgetful map is given in the models of Elkies--Kumar~\cite{Elkies-Kumar} and Bruin--Flynn--Shnidman~\cite{Bruin-Flynn-Shnidman} by $[a:b:c] \mapsto (m,n)$ where
\begin{equation*}
  m = \frac{  ( a - c ) ( a^{2} + a b - 4 a c + b^{2} + b c + c^{2} ) F_1(a,b,c)}{F_2(a,b,c) } \quad \text{and} \quad n = \frac{ -F_3(a,b,c)}{ ( a - c )^{2} F_2(a,b,c)}.
\end{equation*}
Here the homogeneous polynomials $F_i(a,b,c) \in \Z[a,b,c]$ are given by
\begin{flalign*}
  F_1(a,b,c) &= a^{4} - a^{3} b - 8 a^{3} c + 3 a^{2} b c + 18 a^{2} c^{2} - a b^{3} + 6 a b^{2} c + 3 a b c^{2} - 8 a c^{3} + b^{4} - b^{3} c &\\
             &\quad- b c^{3} + c^{4},&\\
  F_2(a,b,c) &= a^{7} - 11 a^{6} c + 39 a^{5} c^{2} - 2 a^{4} b^{3} - 37 a^{4} c^{3} + 10 a^{3} b^{3} c - 37 a^{3} c^{4} - 30 a^{2} b^{3} c^{2} &\\
             &\quad+ 39 a^{2} c^{5} + a b^{6} + 10 a b^{3} c^{3} - 11 a c^{6} + b^{6} c - 2 b^{3} c^{4} + c^{7},&\\
  \intertext{and}
  F_3(a,b,c) &=  a^{9} - 45 a^{8} c + 414 a^{7} c^{2} - 3 a^{6} b^{3} - 1374 a^{6} c^{3} + 36 a^{5} b^{3} c + 1260 a^{5} c^{4} - 99 a^{4} b^{3} c^{2}& \\
             &\quad+ 1260 a^{4} c^{5} + 3 a^{3} b^{6} - 60 a^{3} b^{3} c^{3} - 1374 a^{3} c^{6} + 9 a^{2} b^{6} c - 99 a^{2} b^{3} c^{4} + 414 a^{2} c^{7}& \\
  &\quad+ 9 a b^{6} c^{2} + 36 a b^{3} c^{5} - 45 a c^{8} - b^{9} + 3 b^{6} c^{3} - 3 b^{3} c^{6} + c^{9}.&
\end{flalign*}

\section{On the polynomials \texorpdfstring{$p_D$ and $q_D$}{p{\textunderscore}D and q{\textunderscore}D}}
\label{sec:qD}

It is natural to ask about the geometric significance of the polynomials $p_D$ and $q_D$, i.e., the curves they cut out in $Y_-(D)$.
Here we present some empirical observations, including a conjecture for $q_{40}$, the first case missing from \cref{tab:qD}.

First consider $p_D$ for $D = 5, 8, 12, 13$, and $17$.  For each such $D$, the polynomial $p_D$ is a factor of the discriminant $\Delta(L_D)$ of the original Mestre conic $L_D$, written in terms of $(m,n)$.  In all of the cases except $D=8$, the polynomial $p_D$ is also a factor of $\lambda_D$, viewed as a polynomial in $m$ and $n$.  In the case $D=17$, where $p_D = 1$, we remark that there is a quadratic factor of both $\Delta(L_D)$ and $\lambda_D$ which is a norm from $\Q(m,n)(\sqrt{17})$.  It is not clear how one might identify \emph{a priori} which factor(s) of $\Delta(L_D)$ should contribute to $p_D$.

For $q_D$, the situation is even more mysterious.  Consider the 12 $q_D$'s in \cref{tab:qD} such that $q_D \ne 1$.  In each such case, $q_D$ is not a factor of the original Mestre conic discriminant $\Delta(L_D)$.  When $D=12$ or $D=44$, the lowest degree factor of $q_D$ is a factor of $\Delta(L_D)$.  It is not even apparent during our minim{\iz}ation steps in \cref{alg:main} what $q_D$ should be---it is often not until we diagonal{\iz}e at the end of our minim{\iz}ation process that $q_D$ appears!

For $D = 5, 8, 12$, or $13$, writing $q_D$ in terms of $(m,n)$ gives a rational function which contains $p_D$ as factor of the numerator ($D = 5, 13$) or denominator ($D = 8, 12$).  At least when $D = 5, 8$, the rational functions $q_D(m,n)$ and $p_D(m,n)$ differ by a norm
from $\Q(m,n)(\sqrt D)$.  For example, $q_5 = -6(10g + 3)(15g + 2) = -(m^2 - 5n^2) (m^2 - 5n^2 - 5)$.  Here the curves $g=-\frac 3{10}$ and $g = -\frac 2{15}$ on $Y_-(5)$ are somewhat special: they have no $\Q$-rational points and do not lie in the image of the map from $(m,n)$ to $(g,h)$ (see \cite[Section 5]{CM:RM5} for details).

Now consider the case $D = 21$. Recall from \cite{Elkies-Kumar} that
\begin{align*}
  \lambda_{21} =\,\, & 189g^6 - 594g^5h + 621g^4h^2 - 216g^3h^3 - 378g^4 + 1116g^3h \\
                     &- 954g^2h^2 + 184gh^3 + 16h^4 + 205g^2 - 522gh + 349h^2 - 16,
\end{align*}
and by \Cref{thm:pDs} we have
\begin{equation*}
  q_{21} = 18g^2 - 12gh - 12h^2 - 14.  
\end{equation*}
In Figure~\ref{fig:D21} we plot the real values $(g,h)$ such that $\lambda_{21} = 0$ (red) and $q_{21} = 0$ (blue).

\begin{figure}[t]
  \centering
  \includegraphics[width=0.5\textwidth]{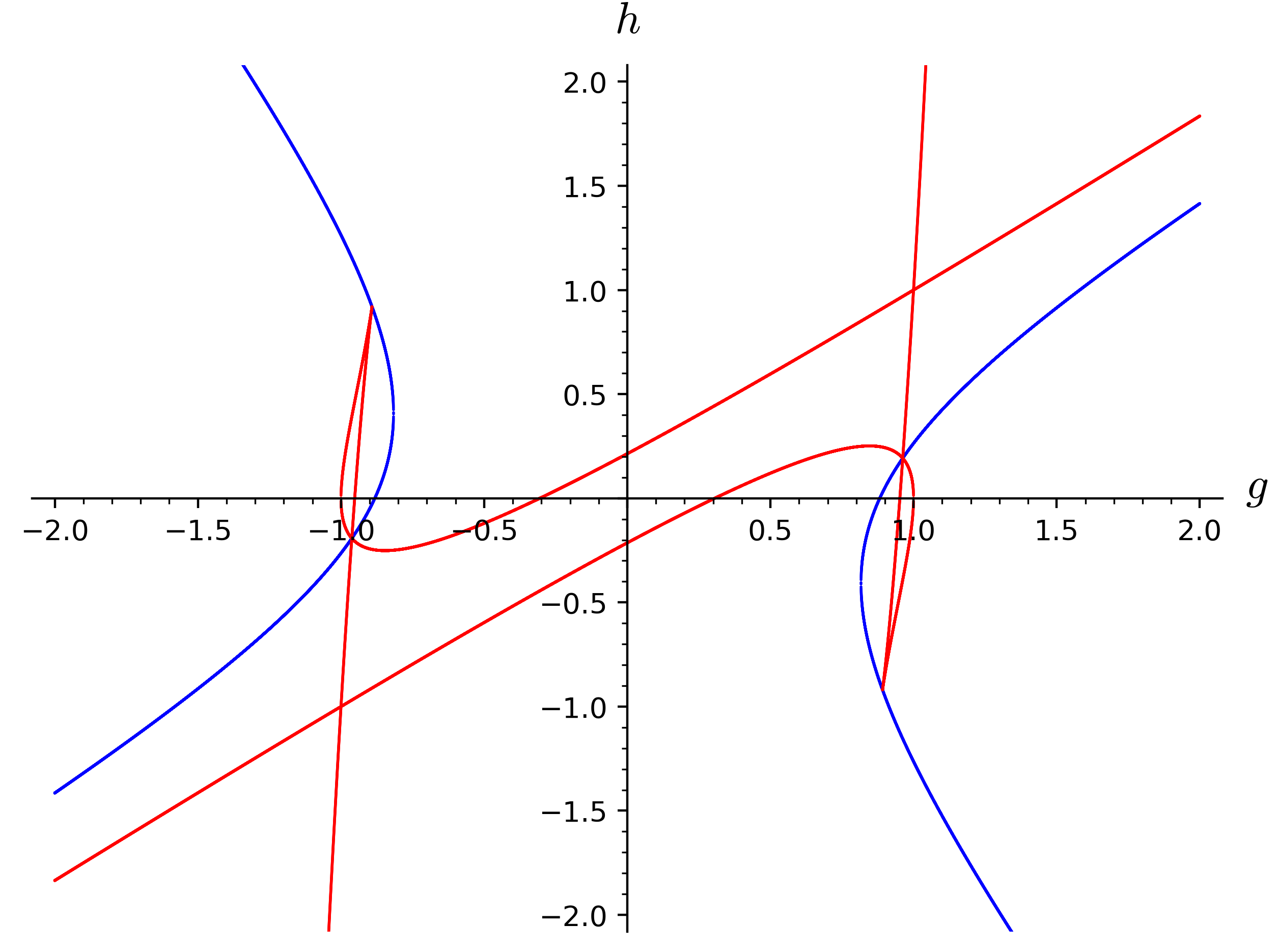}
  \caption{For $D = 21$, the real values $(g,h)$ such that $\lambda_{21} = 0$ (red) and $q_{21} = 0$ (blue).}
  \label{fig:D21}
\end{figure}

We see that there are four points where $\lambda_{21}$ vanishes to order at least $2$ on the curve $q_{21} = 0$: one pair where
\begin{align*}
27g^2 - 25 = 0 \quad&\text{and}\quad 27h^2 - 1 = 0,\\
\shortintertext{and one pair where}
27g^4 + 342g^2 - 289 = 0 \quad&\text{and}\quad 3h^4 + 27h^2 - 25 = 0.
\end{align*}
Indeed, the resultants with respect to $g$ and $h$ are given by 
\begin{align*}
  \operatorname{Res}_g(\lambda_{21}, q_{21}) &= 746496 (27h^2 - 1)^2 (3h^4 + 27h^2 - 25)^2 \\
  \intertext{and}
  \operatorname{Res}_h(\lambda_{21}, q_{21}) &= 64 (27g^2 - 25)^2 (27g^4 + 342g^2 - 289)^2.
\end{align*}
We observed similar behaviour when $D = 24$ and $D = 28$.

When $D = 40$, the singular points of the plane curve defined by the vanishing of the polynomial
\begin{equation*}
  \lambda_{40} = (g^2 - h^2 - 1) (9g^4 - 17g^2h^2 + 8h^4 - 12g^3 + 12gh^2 + 7g^2 - 8h^2 + 10g + 2)
\end{equation*}
occur at the points where
\begin{align*}
  g - 8 = 0 \quad&\text{and}\quad 2h^2 - 125 = 0, \\
  g - 9 = 0 \quad&\text{and}\quad h^2 - 80 = 0,\\
  3g + 1 = 0 \quad&\text{and}\quad h = 0.
\end{align*}

With the ansatz that $q_{40}$ is quadratic in $g$ and $h$ with rational coefficients and imposing that $q_{40}$ vanishes at the points given by the equations above, we see that $q_{40}$ is proportional to
\begin{equation}
  \label{eq:q40}
  -15g^2 + 14h^2 + 10g + 5.
\end{equation}
The expression in \eqref{eq:q40} is not identically a norm in $\Q(\sqrt{40})$ for $g,h \in \Q$, but is a norm for all of the $16$ points $(g,h)$ corresponding to genus $2$ curves given in \cite[Section~17.3]{Elkies-Kumar}, and thus seems to us to be a reasonable guess for an admissible choice of $q_{40}$.


\end{document}